\documentclass[12pt]{amsart}
\usepackage{amssymb}
\usepackage{float}
\usepackage{fullpage}
\usepackage{cleveref}

\newtheorem{theorem}{Theorem}[section]
\newtheorem{lemma}[theorem]{Lemma}
\newtheorem{corollary}[theorem]{Corollary}
\newtheorem{proposition}[theorem]{Proposition}

\theoremstyle{remark}
\newtheorem{remark}[theorem]{Remark}

\numberwithin{equation}{section}

\newcommand{\R}{{\mathbb R}}
\newcommand{\Z}{{\mathbb Z}}
\newcommand{\N}{{\mathbb N}}
\newcommand{\E}{{\mathbb E}}
\newcommand{\dd}{{\rm d}}
\newcommand{\CC}{{\mathbb{C}}}
\renewcommand{\Re}{\mathop{\rm Re}}

\begin{document}

\title[Almost primes]{On Erd\H{o}s sums
of almost primes}

\author{Ofir Gorodetsky}
\address{Mathematical Institute, University of Oxford, Oxford, OX2 6GG, UK}
\email{gorodetsky@maths.ox.ac.uk}

\author{Jared Duker Lichtman}
\address{Department of Mathematics, Stanford University, Stanford, CA, USA}
\email{jared.d.lichtman@gmail.com}

\author{Mo Dick Wong}
\address{Department of Mathematical Sciences, Durham University, Stockton Road, Durham DH1 3LE}
\email{mo-dick.wong@durham.ac.uk}




\maketitle

\begin{abstract}
In 1935, Erd\H{o}s proved that the sums $f_k=\sum_n 1/(n\log n)$, over integers $n$ with exactly $k$ prime factors, are bounded by an absolute constant, and in 1993 Zhang proved that $f_k$ is maximized by the prime sum $f_1=\sum_p 1/(p\log p)$. According to a 2013 conjecture of Banks and Martin, the sums $f_k$ are predicted to decrease monotonically in $k$. In this article, we show that the sums restricted to odd integers are indeed monotonically decreasing in $k$, sufficiently large. By contrast, contrary to the conjecture we prove that the sums $f_k$ increase monotonically in $k$, sufficiently large.

Our main result gives an asymptotic for $f_k$ which identifies the (negative) secondary term, namely $f_k = 1 - (a+o(1))k^2/2^k$ for an explicit constant $a= 0.0656\cdots$. This is proven by a refined method combining real and complex analysis, whereas the classical results of Sathe and Selberg on products of $k$ primes imply the weaker estimate $f_k=1+O_{\varepsilon}(k^{\varepsilon-1/2})$. We also give an alternate, probability-theoretic argument related to the Dickman distribution. Here the proof reduces to showing a sequence of integrals converges exponentially quickly to $e^{-\gamma}$, which may be of independent interest.
\end{abstract}

\section{Introduction}

Let $\Omega(n)$ denote the number of prime factors of $n$, counted with repetition. If $\Omega(n)=k$, $n$ is called a $k$-almost prime. For each $k\ge1$ denote the series
\begin{align*}
f_k = \sum_{\Omega(n)=k}\frac{1}{n\log n}.
\end{align*}
Here $f_k$ is the {\it Erd\H{o}s sum} of $k$-almost primes, also called the $k^\textnormal{th}$ {\it fingerprint number}.
In 1935 Erd\H{o}s \cite{Erd} showed that $f_k = O(1)$ is bounded\footnote{Indeed, his result bounded Erd\H{o}s sums $f(A)=\sum_{n\in A}1/(n\log n)$ uniformly over any primitive set $A$.} and in 1993 Zhang \cite{Zh2} proved the primes have maximal Erd\H{o}s sum, that is, $f_k \le f_1$ holds for all $k$. This had given initial evidence towards the Erd\H{o}s primitive set conjecture, now recently proven by the second author \cite{LEPSC}.

In 2013, Banks and Martin \cite{BM} posed a vast generalization of the Erd\H{o}s primitive set conjecture (see \cite{BM}, \cite{LEPSC} for details and further discussion). In particular, they conjectured that the sums $f_k$ decrease monotonically in $k$. Denote the sum $f_{k,y}$ restricting $f_k$ to integers without prime factors $\le y$, that is,
\begin{align*}
f_{k,y}=\sum_{\substack{\Omega(n)=k\\ p\mid n\Rightarrow p> y}} \frac{1}{n \log n}.
\end{align*}
Banks and Martin further conjectured that for any fixed $y \ge 1$, $f_{k,y}$ decrease monotonically.

We prove that their conjecture holds for $y\ge2$ and $k$ sufficiently large. By contrast, we prove that $f_k=f_{k,1}$ increases monotonically in $k$, sufficiently large.

\begin{theorem}\label{thm:monotonic}
Let $y\ge2$. For $k$ sufficiently large, we have $f_{k-1}<f_{k}$ and $f_{k-1,y}>f_{k,y}$.
\end{theorem}

When $y\ge2$, we believe $f_{k-1,y}>f_{k,y}$ should hold for all $k>1$, in accordance with Banks--Martin \cite{BM}. When $y=1$, we believe $f_{k-1}<f_{k}$ holds for all $k>6$. These inequalities have been verified numerically up to $k\le 20$ \cite{Lalmost}.

The classical Sathe--Selberg theorem gives asymptotics for the counting function of $k$-almost primes, and implies $f_k$ converges to 1 with square-root error. That is $f_k=1+O_{\varepsilon}(k^{\varepsilon-1/2})$, see \cite[Theorem~4.1]{Lalmost}. We give an exponential refinement of this result, which identifies the (negative) secondary term to be $-a_1\, k^2/2^k$ for an explicit constant $a_1= 0.0656\cdots$.

\begin{theorem}\label{thm:fkthm}
For all $k\ge1$ we have
\begin{align*}
f_{k} = 1 - 2^{-k} \big(a_1k^2 + O\big(k \log(k+1)\big)\big),
\end{align*}
where $a_1=(d\log 2)/4$ and
\begin{align}\label{def:d}
d:= \frac{1}{4}\prod_{p > 2} \left( 1-\frac{2}{p}\right)^{-1} \left(1-\frac{1}{p}\right)^2 = 0.37869\cdots.
\end{align}
\end{theorem}

To motivate the proof of Theorem~\ref{thm:fkthm}, we first handle the sifted Erd\H{o}s sums $f_{k,y}$, whose (positive) secondary term is of order $O_y(2^{-k})$ when $y\ge2$, and so converge more rapidly.
\begin{theorem}\label{thm:fkck22k}
Let $y\ge 2$. We have 
\begin{align*}
f_{k,y}
= \prod_{p \le y}\left(1-\frac{1}{p}\right) + a_y/2^k + O_y(k^3/3^{k})
\end{align*}
	uniformly for $k \ge 1$, where $a_y=c_yd_y$ for	\begin{align}
		c_y&= \gamma + \sum_{p \le y} \frac{\log p}{p-1}-\sum_{p>y}\frac{\log p}{(p-1)(p-2)}, \label{def:cy}\\
		d_y	&= d\prod_{2<p \le y} \left(1-\frac{2}{p}\right),
	\end{align}
for $d$ as in \eqref{def:d}, and $\gamma=0.5772\cdots$ is the Euler--Mascheroni constant.
\end{theorem}
As usual $O_y(\cdots)$ means the implicit constant may depend on $y\ge2$, but not on $k$.

In particular, from Theorems \ref{thm:fkthm} and \ref{thm:fkck22k} we infer the even and odd terms in $f_k$ contribute $\frac{1}{2}-(a_1+o(1))k^2/2^k$ and $\frac{1}{2}+(a_2+o(1))2^{-k}$, respectively. As will be shown in the proofs, this discrepancy in the size (and sign) of the secondary terms ultimately come from the different behavior of 
\[ \lim_{s \to 1^+}  (s-1)^{-z} \sum_{p \mid n \Rightarrow p>y} \frac{z^{\Omega(n)}}{n^s}\]
when $y=1$ and $y\ge2$. Namely, the singularity closest to $0$ when $y=1$ is at $z=2$ while for $y\ge2$ it is more distant. This provides a clean answer to the hitherto unexplained numerical observations up to $k\le 20$.
\begin{remark}
Our methods can also handle $\sum_{\omega(n)=k,\, p \mid n \Rightarrow p>y}1/(n\log n)$, where $\omega(n)=\sum_{p \mid n}1$. For this problem, the analysis does not have a discrepancy between $y=1$ and $y\ge2$.

Our methods should also handle the Dirichlet series $\sum_{\Omega(n)=k} n^{-t}$ for $t>1$. For these, striking work of Banks--Martin \cite{BM} shows that $k=1$ is maximal if and only if $t>\tau$, where $\tau= 1.14\cdots$ is the unique solution to a certain functional equation involving the Riemann zeta function. When $t<\tau$ it is not understood which $k=k_t$ is maximal, also see \cite{CLP2,CLPk}. Finally, one may also consider translated sums $\sum_{\Omega(n)=k} 1/(n(\log n+h))$ for $h\in\R$, where the author \cite{Ltranslate} proved that $k=1$ is not maximal for $h>1.04\cdots$. In fact, $k=1$ is minimal if and only if $h>0.803\cdots$.
\end{remark}
\begin{remark}
Hankel contours over the complex plane are used in the Selberg--Delange method, concerning asymptotics for the sums $\sum_{n \le x} z^{\Omega(n)}$, $z$ complex. Saddle point analysis is used in the Sathe--Selberg theorem, which extracts information on the counting function of $k$-almost primes from these sums. These devices are powerful but lead to poor savings. By contrast, our proofs contain neither saddle point analysis nor Hankel contours. 

Moreover, the Selberg--Delange method assumes a zero-free region for $\zeta(s)$, as well as a bound on $\log\zeta(s)$ to the left of $\Re(s)=1$. By contrast, our proofs make no complex-analytic assumptions about $\zeta(s)$ whatsoever. We only use the (very basic) Taylor expansion $\zeta(t)(t-1)= 1 + \gamma(t-1) + O((t-1)^2)$ for real $t\in(1,2)$.
\end{remark}
\subsection{A sequence of integrals}

We also provide an alternate, probability-theoretic argument, which gives (weaker) exponential error $f_k=1 + O(k/2^{k/4})$, but which has potentially much wider applicability to other primitive sets beyond $k$-almost primes. We leave further development of this perspective to future work.

The rough strategy is to first show an asymptotic relation 
\begin{equation}\label{eq:fksim}
f_{k} \sim e^{\gamma} I_{\lfloor k/4\rfloor}
\end{equation}
for a certain sequence of iterated integrals $I_k$. Here $\gamma=0.5772\cdots$ is the Euler--Mascheroni constant.
Specifically, let $I_0=1$ and for $k\ge 1$ let
\begin{align}
I_k=\int_{[0,1]^k} \frac{\dd x_1 \dd x_2 \cdots \dd x_k}{1+x_1(1+x_2(\cdots(1+x_k)\cdots))}.
\end{align}
Let $P_1(n)$ be the largest prime factor of $n$. To connect $f_k$ to $I_k$, we use Mertens' theorem to show that $f_k$ is close to $e^{\gamma}$ times a \textit{weighted average} of $\log P_1(n)/\log n$ over integers $n$ with $\Omega(n)=k$, see \eqref{eq:fk1uj0}. Then for each $n=p_1\cdots p_k$ ($p_1 \ge p_2 \ge \cdots$), the reciprocal of this ratio may be expressed as
\begin{align*}
\frac{\log n}{\log P_1(n)} &=  1 + \frac{\log p_2}{\log p_1}\Big(1+\cdots\Big(1+\frac{\log p_{j+1}}{\log p_j} \Big(1+\frac{\log\big(p_{j+2}\cdots p_k\big)}{\log p_{j+1}}\Big)\Big)\cdots\Big).
\end{align*}
Further, we show the consecutive ratios $\frac{\log p_{i+1}}{\log p_i}\in[0,1]$ are independent and uniformly distributed in $[0,1]$ (with respect to a certain probability measure) which ultimately leads to \eqref{eq:fksim}. See \eqref{eq:fk1uj0} and \eqref{eq:f1kleIj}--\eqref{eq:f1kgeIj} for the precise quantitative formulation of \eqref{eq:fksim}.

Finally, we prove that $I_k$ converges exponentially quickly to $e^{-\gamma}$, which may be of independent interest. The qualitative convergence $I_k\to e^{-\gamma}$ may be deduced from work of Chamayou \cite{Chamayou} (cf.~\cite[Prop.~2.1]{Molchanov}). 

\begin{theorem}\label{thm:iteratedIk}
We have $I_k=e^{-\gamma} + O(2^{-k})$.
\end{theorem}

Our approach to \Cref{thm:iteratedIk} is self-contained and relies on a probabilistic reformulation of the problem, which turns out to be related to the Dickman--Goncharov distribution. See the extensive survey of Molchanov and Panov \cite{Molchanov} for background on this distribution, which appears across probability and theoretical computer science \cite{HT2002}.

\subsection{Proof of monotonicity results}
Here we quickly deduce Theorem~\ref{thm:monotonic} assuming Theorems~\ref{thm:fkthm} and \ref{thm:fkck22k}: Indeed, for large $k$ we have
\begin{align*}
f_k - f_{k-1} &= \frac{\log 2}{4} d \Big((k-1)^2/2^{k-1} - k^2/2^k \Big) \ + \ o(k^2/2^k) \\
&= \frac{\log 2}{4} d k^2/2^k + o(k^2/2^k) \ > \ 0
\end{align*}
by Theorem~\ref{thm:fkthm}. Similarly, by Theorem~\ref{thm:fkck22k} we have
\begin{align*}
f_{k-1,y} - f_{k,y} &= c_yd_y \Big(1/2^{k-1} - 1/2^k \Big) \ + \ o(1/2^k) \\
&= c_yd_y/2^k + o(1/2^k) \ > \ 0.
\end{align*}
Here we used $c_y>0$ for $y\ge2$. Indeed, $c_y$ from \eqref{def:cy} is clearly increasing in $y\ge2$, so $c_y\ge c_2>0$ since
\begin{align*}
\sum_{p>2}\frac{\log p}{(p-1)(p-2)} < .9 < 1.2 < \log 2 + \gamma.
\end{align*}
This completes the proof of Theorem~\ref{thm:monotonic}.

\subsection{Proof method: permutations}
As a short illustration of the proof method, we consider a permutation analogue of the sums $f_k$. This result also may be of independent interest. Namely, define $f_{k,\pi}$ by
\begin{align*}
f_{k,\pi}:=\sum_{m \ge 1} \frac{a_{m,k}}{m}
\end{align*}
where $a_{m,k}$ is the probability that a permutation, chosen uniformly at random from $S_m$, has exactly $k$ disjoint cycles. 
\begin{remark}
To see the analogy with the original sum, given a permutation $\pi \in S_m$ let $C(\pi)$ be the number of disjoint cycles in $\pi$, $d(\pi)$ be $m$ and $|\pi|$ be $|S_m|=m!$. The sum $f_{k,\pi}$ may be expressed as
\[ f_{k,\pi}=\sum_{\substack{C(\pi)=k\\\pi \in \cup_{m \ge 1} S_m}} \frac{1}{|\pi| d(\pi)}.\]
The weight $1/(|\pi| d(\pi))$ is analogous to $1/(n\log n)$. 
\end{remark}
Recall the Riemann zeta function $\zeta(s)=\sum_{n \ge 1}n^{-s}$.

\begin{proposition}
For any $k \ge 1$ we have
	\[f_{k,\pi}= \zeta(k+1)=1+2^{-k-1}+O(3^{-k}).\]
\end{proposition}
\begin{proof}
Consider the exponential generating function of $\sum_{\pi \in S_n} z^{C(\pi)}$:
	\begin{align*}
		F(u,z):=\sum_{m\ge 0} \frac{1}{m!} \left(\sum_{\pi \in S_m} z^{C(\pi)}\right)u^m 
		= 1+\sum_{k\ge1} z^k\sum_{m \ge 1} a_{m,k}u^m.
	\end{align*}
	The exponential formula for permutations shows
	\begin{align}\label{eq:expformpermut}
		F(u,z)=\exp\left(\sum_{n \ge 1}z \frac{u^n}{n}\right)=\exp\left(-z\log(1-u)\right).
	\end{align}
	The Taylor series for $\exp(-z\log(1-u))$ in $z$ is 
	\[ \exp\left(-z\log(1-u)\right)= \sum_{k \ge0} \frac{z^k}{k!}\left(-\log(1-u)\right)^k,\]
	so extracting the coefficient of $z^k$ in $F(u,z)$ yields, by \eqref{eq:expformpermut},
	\begin{equation}\label{eq:zk coeff}
		\sum_{m \ge 1}a_{m,k}u^m = \frac{\left(-\log (1-u)\right)^k}{k!}.
	\end{equation}
Since $1/m = \int_0^1 u^{m-1}\,\dd{u}$, we can integrate \eqref{eq:zk coeff} over $u\in[0,1]$ to obtain
	\[f_{k,\pi}=\sum_{m \ge 1 }\frac{a_{m,k}}{m}
	= \int_{0}^{1}\frac{\left(-\log(1-u)\right)^k}{k!}\frac{\dd u}{u} 
	= \int_{0}^{\infty} \frac{x^k}{\Gamma(k+1)} \frac{\dd x}{e^x-1}
	= \zeta(k+1).\]
	This completes the proof. Here in the third equality we performed the change of variables $u=1-e^{-x}$, and in the fourth we applied Riemann's famous integral representation for $\zeta(s)$, with $s=k+1$,
	\[ \Gamma(s)\zeta(s)=\int_{0}^{\infty}\frac{x^{s-1}}{e^x-1}\dd x.\]
\end{proof}
\begin{remark}
In the same way that we shall relate $f_k$ to the sequence of integrals $I_k$ in Theorem~\ref{thm:iteratedIk}, one can relate $f_{k,\pi}$ to $I_k$ as well.
\end{remark}
\section{Sifted sums}
\subsection{Preparation}
We now collect a few analytic properties of the generating functions that will play central roles in the proof of Theorem~\ref{thm:fkck22k}. Let us start with the function
\begin{align}\label{def:Fsz}
F(s,z) := \sum_{n\ge 1} \frac{z^{\Omega(n)}}{n^s} = \prod_p \Big(1-\frac{z}{p^s}\Big)^{-1}, \qquad s \in \mathbb{R}.
\end{align}
\begin{lemma}\label{lem:cont}
The function $F(s,z)$ converges absolutely for $s>1$ and $|z|<2$. In this domain $F$ has a power series representation in $z$ and it defines a smooth function in $s$.\footnote{In this paper we will only need first and second derivatives with respect to $s$, which is always seen as a real-valued variable.}
\end{lemma}
\begin{proof}
Observe that for $s \ge 1+\varepsilon$ and $|z| \le 2^{-\varepsilon}$ the $p$th factor in \eqref{def:Fsz} is
\[\Big(1-\frac{z}{p^s}\Big)^{-1} = \exp\left(\sum_{i \ge 1}\frac{z^i}{ip^{si}}\right) = \exp\left(O_{\varepsilon}(|z|p^{-s})\right)\]
and so the product in \eqref{def:Fsz} is a uniform limit of power series in $z$ which are smooth in $s$.
\end{proof}
For $s>1$ and $|z|<2$ we define
\begin{align*}
	G(s,z) := F(s,z) (s-1)^z.
\end{align*}
The following lemma extends the range of definition of $G$. For any $y \ge 1$ denote by $y_1$ the smallest prime greater than $y$. 
\begin{lemma}
For every prime $q$, $\prod_{p \le q}(1-z/p^s)G(s,z)$, as well as its derivatives in $s$, are smooth in $s \ge 1$ and have a power series expansion in $z$ with radius $q_1$. Moreover, for $s \ge 1$ and $k \ge 0$, the $k$th derivative of $G(s,z)$ with respect to $s$ has a meromorphic continuation to $z \in \CC$ with poles of order $k$ at $z=p^s$ for every prime $p$. This continuation satisfies
\[ G(1,z) = \prod_{p} \left(1-\frac{z}{p}\right)^{-1} \left(1- \frac{1}{p}\right)^z.\]
\end{lemma}
\begin{proof}
We write $G$ as 
\begin{equation}\label{eq:GF} 
	G(s,z) = F(s,z) \zeta(s)^{-z}  (\zeta(s)(s-1))^z .
\end{equation}
It is well known that $\lim_{s \to 1^+}\zeta(s)(s-1)=1$, and that extending $\zeta(s)(s-1)$ to $s=1$, by setting it to equal $1$ there, it is a smooth function in $s \ge 1$.\footnote{Throughout we identify $\zeta(s)(s-1)$ at $s=1$ with $1$. A function $f$ is smooth on $[1,\infty)$ if it belongs to $\cap_{k=0}^{\infty}C^k([1,\infty))$ where for $k \ge 1$, $C^k([1,\infty)):=\{ f\colon [1,\infty) \to \mathbb{R}, f \text{ is differentiable}, f' \in C^{k-1}([1,\infty))\}$ and differentiability at $1$ is defined via the right derivative. For $k=0$, $C^0([1,\infty))$ consists of continuous functions on $[1,\infty)$, with continuity at $1$ defined by right-continuity.} Hence, $(\zeta(s)(s-1))^z=\exp(z \log (\zeta(s)(s-1))$ is smooth in $s \ge 1$ and has a power series representation in $z$ with infinite radius of convergence:
\begin{equation}\label{eq:inf} (\zeta(s)(s-1))^z = \sum_{i \ge 0} \frac{z^i}{i!}(\log (\zeta(s)(s-1)))^i.
\end{equation}
At $s=1$, \eqref{eq:inf} is equal to $1$. It remains to consider  $F(s,z)\zeta(s)^{-z}$. It has the following Euler product:
\begin{equation}\label{eq:FE}
F(s,z)\zeta(s)^{-z} = \prod_{p} \left( 1-\frac{z}{p^s}\right)^{-1} \left(1-\frac{1}{p^s}\right)^z.
\end{equation}
For every prime $q$, we can use the product in \eqref{eq:FE} to define $F(s,z)\zeta(s)^{-z}$ as a product of the rational function $\prod_{p \le q}(1-z/p^s)^{-1}$ (which has simple poles in $z$ at $z=p^s$ for every prime $p \le q$, and is defined for $s\ge 1$) with a function that has a power series representation in $z$ with radius of convergence $q_1$ and is defined for $s \ge 1$.  This is because if $|z| < q_1^{1-\varepsilon}$  and $p>q$ then the $p$th term in the right-hand side of \eqref{eq:FE} equals
\[\exp\left( \sum_{i=2}^{\infty} 
 \frac{z^i-z}{ip^{si}}\right) = \exp \left(O_{\varepsilon}\left((|z|^2+|z|)p^{-2 s}\right)\right),\] and the product over $p>q$ converges absolutely and uniformly. 
\end{proof}
We also introduce
\begin{align}\label{def:Fszy}
F_y(s,z) := \sum_{\substack{n \ge 1\\ p\mid n\Rightarrow p> y}} \frac{z^{\Omega(n)}}{n^s} = \prod_{p >y} \left(1-\frac{z}{p^s}\right)^{-1} = \ \prod_{p \le y} \left(1-\frac{z}{p^s}\right)\,F(s,z),
\end{align}
and 
\begin{equation}\label{def:Gszy}
\begin{split}
G_y(s,z) &:= \prod_{p \le y} \left(1-\frac{z}{p^s}\right)\,G(s,z) \\
&= \prod_{p\le y}\Big(1-\frac{1}{p^s}\Big)^{z}\prod_{p>y} \Big(1-\frac{z}{p^s}\Big)^{-1}\Big(1-\frac{1}{p^s}\Big)^{z}(\zeta(s)(s-1))^z.
\end{split}
\end{equation}
For any smooth function $H(s,z)$, we shall denote by $H^{(a,b)}$ the mixed partial derivative \[H^{(a,b)} =\frac{\partial^{a+b}}{\partial s^a \partial z^b} H.\]
For every $j \ge 0$ and fixed $s\ge 1$, $G^{(j,0)}_y(s,z)$ has a meromorphic continuation to $\CC$ with poles at $z=p^s$ for every prime $p>y$ by Lemma~\ref{lem:cont}. Taking the logarithmic derivative of \eqref{def:Gszy} with respect to $s$ gives
\begin{align*}
\frac{G_y^{(1,0)}}{G_y}(s,z)=&\frac{\partial}{\partial s}\big[\log G_y(s,z)\big]
=\frac{\partial}{\partial s}
\bigg[\sum_{p\le y}z\log\Big(1-\frac{1}{p^s}\Big)\\
&+\sum_{p\ge y_1}\bigg(z\log\Big(1-\frac{1}{p^s}\Big)- \log\Big(1-\frac{z}{p^s}\Big)\bigg) + z\log \big(\zeta(s)(s-1)\big)\bigg]
\end{align*}
so that
\begin{lemma}
We have
\begin{align}\label{eq:dlogGy}
G_y^{(1,0)}(s,z)& = z G_y(s,z)\left(\sum_{p \le y}\frac{\log p}{p^s-1} -\sum_{p\ge y_1}\frac{(z-1)\log p}{(p^s-1)(p^s-z)} + (\log (\zeta(s)(s-1)))'\right).
\end{align}
\end{lemma}

\begin{remark}
These generating functions and their values have natural connections to related questions about $k$-almost primes \cite{BKL,Ldissect}. In particular $G_2(1,2)=\tfrac{1}{4}\prod_{p>2}(1-\tfrac{2}{p})^{-1}(1-\tfrac{1}{p})^2 = d$ as in \eqref{def:d} equals $2\beta_2$, from the main term of \cite[Theorem~1.2]{BKL}.
\end{remark}

\subsection{Proof of Theorem \ref{thm:fkck22k}}\label{sec:proof}
Let $y \ge 1$, and let $y_1$ be the smallest prime greater than $y$. We have the integral representation
\begin{equation}\label{eq:intrep}
f_{k,y}=\sum_{\substack{\Omega(n)=k\\ p\mid n\Rightarrow p> y}} \frac{1}{n \log n}=\int_{1}^{\infty} \sum_{\substack{\Omega(n)=k\\p\mid n\Rightarrow p> y}} n^{-s} \dd s.
\end{equation}
The smallest number $n$ with $\Omega(n) = k$ and all prime factors greater than $y$ is $y_1^k$, so the contribution of $s \ge 2$ to \eqref{eq:intrep} is
\[ \int_{2}^{\infty} \sum_{\substack{\Omega(n)=k\\ p\mid n\Rightarrow p> y} }n^{-s} \dd s  = \sum_{\substack{\Omega(n)=k\\ p\mid n\Rightarrow p> y}} \frac{1}{n^{2}\log n}  \le \sum_{n \ge y_1^k} \frac{1}{n^{2}\log n} \ll \frac{1}{k} \int_{y_1^k - 1}^{\infty} \frac{dt}{t^{2}} \ll  \frac{1}{k y_1^k}. \]
Thus we have
\begin{equation}\label{eq:fkyiky}
	f_{k,y} = I_{k,y} + O\left(\frac{1}{k y_1^{k}}\right),
\end{equation}
where $I_{k,y}$ is the corresponding integral over $s\in[1,2]$, namely,
\begin{equation}\label{def:Iky}
	I_{k,y} := \int_{1}^{2} \sum_{\substack{\Omega(n)=k\\ p\mid n\Rightarrow p> y}} n^{-s}\dd s
	=\int_{1}^{2}\frac{1}{k!} F_y^{(0,k)}(s,0)\dd s.
\end{equation}
Here we term-wise differentiated $F_y(s,z)$ in \eqref{def:Fszy} with respect to $z$, for $s > 1$.

Next from the Taylor series
\begin{equation}\label{eq:s1z}
(s-1)^{-z} = \exp\left(-z\log(s-1)\right) = \sum_{i \ge0} \frac{\left(-\log(s-1)\right)^i}{i!}z^i,
\end{equation}
we apply the product rule to $F_y(s,z)=(s-1)^{-z}\,G_y(s,z)$, giving
\[ \frac{1}{k!}F_y^{(0,k)}(s,0) = \sum_{i=0}^k \frac{\left(-\log(s-1)\right)^{k-i}}{(k-i)!} \frac{1}{i!} G_y^{(0,i)}(s,0).\]

Thus \eqref{def:Iky} becomes
\begin{equation}\label{eq:Ikintegral}
	I_{k,y} =\sum_{i=0}^{k} \int_{1}^{2} \frac{\left(-\log(s-1)\right)^{k-i}}{(k-i)!} \frac{1}{i!}G_y^{(0,i)}(s,0) \dd s.
\end{equation}
Now we introduce a similar integral $I'_{k,y}$, given by evaluating $G^{(0,i)}(s,0)$ in the integrand at $s=1$, namely,
\begin{equation}\label{eq:Ik1integral}
	I'_{k,y} :=\sum_{i=0}^{k} \frac{1}{i!}G_y^{(0,i)}(1,0) \int_{1}^{2} \frac{\left(-\log(s-1)\right)^{k-i}}{(k-i)!} \dd s.
\end{equation}
To handle $I'_{k,y}$, we substitute $s=1+e^{-t}$ and obtain, for any $j\ge 0$,
\begin{equation}\label{eq:gamma}
 \int_{1}^{2} \frac{\left(-\log(s-1)\right)^{j}}{j!}\dd s = \int_{0}^{\infty} \frac{e^{-t}t^{j}}{j!}\dd t = \frac{\Gamma(j+1)}{j!}=1,
\end{equation}
as we are evaluating the Gamma function at $j+1$. Hence \eqref{eq:Ik1integral} simplifies as
\begin{equation}\label{eq:Ik1simpler}
	I'_{k,y} = \sum_{i=0}^{k}  \frac{1}{i!}G_y^{(0,i)}(1,0).
\end{equation}
In the upcoming subsections, we shall estimate $I_{k,y}$ by means of the following lemmas for $I'_{k,y}$ and $I_{k,y}-I'_{k,y}$. 
\begin{lemma}\label{lem:I1minus1y}
	Let $y \ge 1$. We have $I'_{k,y}= G_y(1,1) + O_y(y_1^{-k})$.
\end{lemma}
\begin{lemma}\label{lem:IminusI2y}
	Let $y \ge 2$. We have $I_{k,y}=I'_{k,y} + G_y^{(1,0)}(1,2)/2^{k+1}+ O_y(k^3/3^{k})$.
\end{lemma}

\begin{proof}[Proof of Theorem \ref{thm:fkck22k} assuming Lemmas \ref{lem:I1minus1y} and \ref{lem:IminusI2y}]
Recalling \eqref{eq:fkyiky}, we have
\begin{equation}\label{eq:fkyGy}
\begin{split}
f_{k,y} & = I_{k,y} + O(1/(k y_1^{k})) \\
& = I'_{k,y} + G_y^{(1,0)}(1,2)/2^{k+1}+ O_y(k^3/3^{k}) \\
& = G_y(1,1)  + G_y^{(1,0)}(1,2)/2^{k+1}+ O_y(k^3/3^{k})
\end{split}
\end{equation}
for $y \ge 2$. To compute the constants above, we first note $G_y(1,1) = \prod_{p\le y}(1-1/p)$. Next, by \eqref{def:Gszy} and \eqref{eq:dlogGy} with $(s,z)=(1,2)$ we have
\begin{equation}\label{eq:Gy10}
\begin{split}
G_y(1,2)&=\prod_{p \le y} \left(1-\frac{1}{p}\right)^2 \prod_{p \ge y_1} \left(1-\frac{2}{p}\right)^{-1} \left(1-\frac{1}{p}\right)^2 = d_y, \\
G_y^{(1,0)}(1,2) &=2G_y(1,2) \left(\sum_{p \le y}\frac{\log p}{p-1} -\sum_{p\ge y_1}\frac{\log p}{(p-1)(p-2)} + \gamma\right) = 2d_yc_y
\end{split}
\end{equation}
for $y \ge 2$. Here we used $(\log (\zeta(s)(s-1)))'|_{s=1}=\gamma$. Hence plugging \eqref{eq:Gy10} and $G_y(1,1) = \prod_{p\le y}(1-1/p)$ back into \eqref{eq:fkyGy}, we conclude
\begin{align}
f_{k,y} & = \prod_{p \le y}\left(1-\frac{1}{p}\right) + c_yd_y/2^{k}+ O_y(k^3/3^{k}).
\end{align}
\end{proof}

\subsection{Proof of Lemma \ref{lem:I1minus1y}}
We use the notation $[z^n]A(z)=\frac{1}{n!} (\frac{d^n}{dz^n}A)(0)$ to denote the coefficient of $z^n$ in $A$, where $A$ is a function with Taylor series representation at $z=0$.
From the representation of $I'_{k,y}$ in \eqref{eq:Ik1simpler}, we have for all $k \ge 1$,
\[ I'_{k,y} =[z^k]  \frac{G_y(1,z)}{1-z} = \frac{1}{2\pi i} \int_{|z|=1/2}\frac{G_y(1,z)}{1-z} \frac{\dd z}{z^{k+1}},\]
using Cauchy's integral formula. Here the integral ranges over a circle centered around $z=0$, oriented counterclockwise, with radius $1/2$.

The function $G_y(1,z)$ has simple poles at $z=p$ for every prime $p>y$; these are its only poles. The rational function $1/(1-z)$ has a simple pole at $z=1$. So recalling the smallest prime $y_1>y$, the only poles of $G_y(1,z)/(1-z)$ in the range $1/2<|z|<y_1+1/2=:R$ occur at $z=1$ and $z=y_1$. Thus by Cauchy's residue theorem,
\begin{align}\label{eq:I1lem3.1}
I'_{k,y}&= \frac{1}{2\pi i} \int_{|z|=R} \frac{G_y(1,z)}{1-z} \frac{\dd z }{z^{k+1}} + G_y(1,1) -\frac{\lim_{z \to y_1} (z-y_1)\,G_y(1,z)}{(1-y_1)y_1^{k+1}}.
\end{align}

Note that $\lim_{z \to y_1} G_y(1,z)(z-y_1) \ll_{y} 1$. Then we claim $|G_y(1,z)|\ll_y 1$ in the integrand of \eqref{eq:I1lem3.1}, from which we conclude
\begin{align}\label{eq:I1lem3.11}
I'_{k,y}&= \int_{|z|=R}\frac{O_y(1)}{R-1}\frac{\dd z }{R^{k+1}} + G_y(1,1) - \frac{O_y(1)}{(1-y_1)y_1^{k+1}}
 \ = \ G_y(1,1) + O_y(y_1^{-k}).
\end{align}
To show this claim, note that if $|z|=R$ and $p>2R$, then
\begin{align*}
\left| \left(1-\frac{1}{p}\right)^z\left(1-\frac{z}{p}\right)^{-1} \right| = \bigg|\exp\bigg( \sum_{i \ge 2} \frac{z^i-z}{ip^i} \bigg)\bigg| \le \exp\bigg( \sum_{i \ge 2} \frac{|z|^i}{p^i} \bigg)\le 
\exp\Big(\frac{2R^2}{p^2}\Big).
\end{align*}
Hence, recalling \eqref{def:Gszy} with $s=1$ we obtain
\begin{align*}
\max_{|z|=R}\left|G_y(1,z)\right| & = \max_{|z|=R} \left| \prod_{p\le y}\left(1-\frac{1}{p}\right)^{z}\prod_{p\ge y_1}\left(1-\frac{1}{p}\right)^z\left(1-\frac{z}{p}\right)^{-1}\right| \\
&\le  \prod_{p \le y} \left(1-\frac{1}{p}\right)^{-R} \prod_{p > 2R} \exp\Big(\frac{2R^2}{p^2}\Big) \, \cdot\max_{|z|=R} \prod_{y_1\le p\le 2R} \left|\left(1-\frac{1}{p}\right)^z\left(1-\frac{z}{p}\right)^{-1}\right| 
\ \ll_y \  1.
\end{align*}
This proves the claim, and hence Lemma~\ref{lem:I1minus1y} follows.

\subsection{Proof of Lemma \ref{lem:IminusI2y}}
By Taylor expansion at $s=1$, we have uniformly for $s \in [1,2]$
\[G_y^{(0,i)}(s,0) = G_y^{(0,i)}(1,0) + (s-1) b_{i} + O\big( (s-1)^{2} c_{i}\big),\]
for coefficients
\begin{equation}\label{eq:bici} 
	b_{i} :=G_y^{(1,i)}(1,0)\qquad\text{and}\qquad
	c_{i} :=\max_{s'\in [1,2]} \left| G_y^{(2,i)}(s',0)\right|.
\end{equation}
Thus subtracting \eqref{eq:Ikintegral} from \eqref{eq:Ik1integral}, we have
\begin{equation}\label{eq:IminusI1}
\begin{split}
	 I_{k,y}-I'_{k,y}   &= \sum_{i=0}^{k} \int_{1}^{2} \frac{\left(-\log(s-1)\right)^{k-i}}{(k-i)!} \frac{1}{i!} \left( G_y^{(0,i)}(s,0) - G_y^{(0,i)}(1,0)\right) \dd s  \\
	& = \sum_{i=0}^{k} \int_0^1 \frac{\left(-\log s\right)^{k-i}}{(k-i)!i!} \left( sb_{i} + O( s^{2} c_{i})\right) \dd s.
 \end{split}
\end{equation}
Substituting $s=e^{-t}$ shows that
\begin{equation}\label{eq:IminusI1bi}
	\begin{split}
		\sum_{i=0}^{k}\frac{b_{i}}{i!(k-i)!}\int_0^1 (-\log s)^{k-i}s \dd s
		&= \sum_{i=0}^{k}\frac{b_i}{i!(k-i)!}\int_{0}^{\infty} t^{k-i}e^{-2t} \dd t  \\
		&=\sum_{i=0}^{k}\frac{b_{i} 2^{i-k-1}}{i!(k-i)!} \int_{0}^{\infty} v^{k-i}e^{-v} \dd v \
		= \ \sum_{i=0}^{k} \frac{b_{i}}{i!} 2^{i-k-1}
	\end{split}
\end{equation}
and similarly
\begin{align}\label{eq:IminusI1ci}
\sum_{i=0}^{k}\frac{c_{i}}{i!(k-i)!} \int_0^1(-\log s)^{k-i} s^2 \dd s = \sum_{i=0}^{k} \frac{c_{i}}{i!}3^{i-k-1}.
\end{align}
Plugging \eqref{eq:IminusI1bi} and \eqref{eq:IminusI1ci} back into \eqref{eq:IminusI1} gives
\begin{align}
I_{k,y}-I'_{k,y} &= \sum_{i=0}^{k}\bigg( \frac{b_{i}}{i!} 2^{i-k-1} + O\Big(\frac{c_{i}}{i!}3^{i-k-1}\Big)\bigg).
\end{align}

So proceeding as in the proof of Lemma~\ref{lem:I1minus1y} (as in \eqref{eq:I1lem3.1}), by Cauchy's integral formula and residue theorem,
\begin{equation}\label{eq:zkGy101z}
\begin{split}
\sum_{i=0}^{k} \frac{b_{i}}{i!} 2^{i-k-1} &= 2^{-1}[z^k] \frac{G_y^{(1,0)}(1,z)}{1-\frac{z}{2}}
= \frac{1}{4\pi i} \int_{|z|=1/2} \frac{G_y^{(1,0)}(1,z)}{1-\frac{z}{2}} \frac{\dd z }{z^{k+1}} \\
&= \frac{1}{2\pi i} \int_{|z|=R} \frac{G_y^{(1,0)}(1,z)}{2-z} \frac{\dd z }{z^{k+1}} + G_y^{(1,0)}(1,2)2^{-k-1} - \frac{\lim_{z \to y_1} (z-y_1)\,G_y^{(1,0)}(1,z)}{(2-y_1)y_1^{k+1}} \\
&= G_y^{(1,0)}(1,2)2^{-k-1} + O_y(y_1^{-k})
\end{split}
\end{equation}
holds where $R:=y_1+1/2$.

Finally, we claim $c_i \ll_y i! (i+1)^2/y_1^i$, in which case \[\sum_{i=0}^{k} \frac{c_{i}}{i!}3^{i-k-1}\ll \sum_{i=0}^{k}y_1^{-i}(i+1)^2 3^{i-k-1} \ll \sum_{i=0}^{k}(i+1)^2\,3^{-k} \ll_y k^3/3^{k}\] since $y_1\ge 3$ (as $y\ge2$). Thus combined with \eqref{eq:zkGy101z}, we conclude
\begin{align}
I_{k,y}-I'_{k,y} = G_y^{(1,0)}(1,2)/2^{k+1} + O_y(k^3/3^{k}).
\end{align}

Hence to complete the proof of Lemma~\ref{lem:IminusI2y}, it suffices to show $c_i \ll_y i! (i+1)^2/y_1^i$, which by definition means that uniformly for $s \in [1,2]$,
\begin{align}\label{eq:cibound}
[z^i] G_y^{(2,0)}(s,z) \ll_y (i+1)^2/y_1^i.
\end{align}
To this end, recall $G_y^{(1,0)}(s,z) = zG_y(s,z) c(s)$ by \eqref{eq:dlogGy}, where
\begin{align*}
c(s)=c_y(s,z) := \sum_{p \le y}\frac{\log p}{p^s-1} -\sum_{p\ge y_1}\frac{(z-1)\log p}{(p^s-1)(p^s-z)} + (\log (\zeta(s)(s-1)))'.
\end{align*}
So differentiating again with respect to $s$ we obtain
\begin{equation}\label{eq:Gy20}
\begin{split}
G_y^{(2,0)}(s,z)
&= zG_y^{(1,0)}(s,z)c(s) + zG_y(s,z)c'(s)  \\
&=G_y(s,z) \big(z^2c(s)^2 + zc'(s)\big).
\end{split}
\end{equation}
The derivative of $c$ with respect to $s$ is
\begin{align*}
c'(s) 
&= -\sum_{p \le y}\frac{ p^s (\log p)^2}{(p^s-1)^2} \ + \ \sum_{p\ge y_1}\frac{(z-1)(\log p)^2(2p^{2s}-(z+1)p^s)}{(p^s-1)^2(p^s-z)^2} \ + \ \big(\log (\zeta(s)(s-1))\big)''.
\end{align*}
For fixed $s \in [1,2]$, note $c$ and $c'$ are meromorphic functions on $\CC$, with poles located only at $z=p^s$ for each $p \ge y_1$. Thus
\begin{align*}
	G_y(s,z) &=\left(1-\frac{z}{y_1^s}\right)^{-1} H_{y,1}(s,z),\\ 
	z^2c(s)^2 + zc'(s) &=  \left(1-\frac{z}{y_1^s}\right)^{-2} H_{y,2}(s,z),
\end{align*}
for functions $H_{y,1}, H_{y,2}$, whose smallest pole is at $z=y_2^s$ where $y_2$ is the smallest prime larger than $y_1$.

Letting $H_{y}:=H_{y,1}H_{y,2}$, we see \eqref{eq:Gy20} becomes
\begin{equation}\label{eq:Gy201}
	G_y^{(2,0)}(s,z) = \left(1-\frac{z}{y_1^s}\right)^{-3} H_{y}(s,z).
\end{equation}
Note $H_{y}$ has no poles inside $|z| \le y_2^{s-\varepsilon}$ so  $\max_{s\in[1,2]}|H_{y}(s,z)| \ll_{y,\varepsilon} 1$ uniformly for $|z| \le y_2^{s-\varepsilon}$, as we take the maximum of the continuous function $|H_y(s,z)|$ over the compact set $\{(s,z): 1 \le s \le 2, \, |z| \le y_2^{s-\varepsilon}\}$.

. Thus by Cauchy's integral formula, 
\[ [z^i] H_{y}(s,z) = \frac{1}{2\pi i} \int_{|z|=y_2^{s-\varepsilon}} \frac{H_{y}(s,z)}{z^{i+1}}\dd z \ll y_2^{-i(s-\varepsilon)}\max_{|z|=y_2^{s-\varepsilon}}|H_y(s,z)| \ll_{y,\varepsilon} y_2^{-i(1-\varepsilon)}\]
uniformly for $s \in [1,2]$. By the binomial theorem,
\[ [z^i]\left(1-\frac{z}{y_1^{s}}\right)^{-3} = y_1^{-is} \binom{i+2}{2} \ll (i+1)^2/y_1^{i}\]
uniformly for $s \in [1,2]$ and $i \ge 0$. Hence by the product rule, from \eqref{eq:Gy201} we conclude
\[ [z^i] G_y^{(2,0)}(s,z) =\sum_{i_1+i_2=i}  [z^{i_1}]\left(1-\frac{z}{y_1^{s}}\right)^{-3} [z^{i_2}] H_y(s,z)\ll_y  (i+1)^2/y_1^{i}.\]
This gives \eqref{eq:cibound} as desired, which completes the proof.

\begin{remark}
By a similar proof as of \eqref{eq:cibound} above, for any $y\ge2$, $m\ge0$,
\begin{align}\label{eq:cimbound}
[z^i] G_y^{(m,0)}(s,z) \ll_{m,y} (i+1)^m/y_1^i
\end{align}
holds uniformly for $s \in [1,2]$ and $i \ge 0$.
\end{remark}
\section{Proof of Theorem \ref{thm:fkthm}}
Recall $f_k = f_{k,1}$, $F=F_1$ and $G=G_1$. By \eqref{eq:fkyiky} with $y=1$, we have
\[ f_k = I_k + O\left(\frac{1}{k2^k}\right)\]
where
\[ I_{k}= \int_{1}^{2}\frac{1}{k!} F^{(0,k)}(s,0)\dd s.\]
We apply the product rule to 
\[F(s,z)=(s-1)^{-z}G(s,z)=(s-1)^{-z} \left(1-\frac{z}{2^s}\right)^{-1}G_2(s,z),\]
giving
\begin{align*}
\frac{1}{k!}F^{(0,k)}(s,0) = \sum_{i+j+l = k} \frac{\left(-\log(s-1)\right)^{l}}{l!}2^{-js}  \frac{1}{i!} G_2^{(0,i)}(s,0),
\end{align*}
using the power series in \eqref{eq:s1z}, as well as
$\left(1-\frac{z}{2^{s}}\right)^{-1} = \sum_{j \ge 0} 2^{-js}z^j$.
Thus we have
\begin{equation}\label{eq:Ikintegral2}
I_{k} =\sum_{i+j+l= k} \int_{1}^{2} \frac{\left(-\log(s-1)\right)^{l}}{l!}2^{-js}  \frac{1}{i!} G_2^{(0,i)}(s,0) \dd s.
\end{equation}
Now we introduce a similar integral $I'_{k}$, given by evaluating $G_2^{(0,i)}(s,0)$ in the integrand at $s=1$, namely,
\begin{align}\label{eq:Ik1integral2}
I'_{k} &=\sum_{i+j+l= k} \int_{1}^{2} \frac{\left(-\log(s-1)\right)^{l}}{l!}2^{-js}  \frac{1}{i!} G_2^{(0,i)}(1,0) \dd s.
\end{align}

Hence to establish Theorem~\ref{thm:fkthm}, it suffices to prove the following lemmas for $I'_k$ and $I_{k}-I'_{k}$.
\begin{lemma}\label{lem:I1minus1}
We have $I'_k = 1 -  \frac{\log 2}{4}2^{-k} \big(dk^2 +O(k\log (k+1))\big)$.
\end{lemma}
\begin{lemma}\label{lem:IminusI2}
We have $I_{k}=I'_{k} + O(k/2^{k})$.
\end{lemma}
These are the $y=1$ analogues of Lemmas~\ref{lem:I1minus1y} and \ref{lem:IminusI2y}.

\subsection{Proof of Lemma \ref{lem:IminusI2}}
By the mean value theorem we have, uniformly for $s \in [1,2]$,
\[ \Big|G_2^{(0,i)}(s,0) - G_2^{(0,i)}(1,0)\Big| \le (s-1)b_i\]
for coefficients
\begin{align*}
b_{i} :=\max_{s'\in [1,2]} \Big| G_2^{(1,i)}(s',0) \Big|.
\end{align*}
Thus subtracting \eqref{eq:Ikintegral2} from \eqref{eq:Ik1integral2}, we find
\begin{equation}\label{eq:IminusI12}
\begin{split}
\big| I_{k}-I'_{k} \big| &=\left| \sum_{i+j+l = k} \int_{1}^{2} \frac{\left(-\log(s-1)\right)^l}{l!}2^{-js} \frac{1}{i!} \left( G_2^{(0,i)}(s,0) - G_2^{(0,i)}(1,0)\right) \dd s  \right| \\
& \le \sum_{i+j+l = k} \int_{1}^{2}  \frac{\left(-\log(s-1)\right)^l}{l!}2^{-js}  (s-1) \frac{b_i}{i!} \dd s.
\end{split}
\end{equation}
By \eqref{eq:cimbound}, we have uniformly for $t\in[1,2]$,
\begin{align}\label{eq:bibound} 
\frac{1}{i!}G_2^{(1,i)}(t,0) = [z^i]G_2^{(1,0)}(t,z)  \ll (i+1)3^{-i}.
\end{align}
Hence $\frac{b_i}{i!}  \ll (i+1)3^{-i}$, so that \eqref{eq:IminusI12} implies
\begin{align*}
I_k - I'_k &\ll \sum_{i+j+l= k}  \int_{0}^{\infty} \frac{t^l}{l!}  2^{-j(1+e^{-t})} e^{-2t} (i+1) 3^{-i}  \dd t\\
&\le \sum_{i+j+l= k} (i+1)3^{-i}2^{-j} \int_{0}^{\infty} \frac{t^l}{l!} e^{-2t}   \dd t\\
& = \sum_{i+j+l= k} (i+1) 3^{-i} 2^{-j-l-1}\int_{0}^{\infty} \frac{u^l}{l!}e^{-u} \dd u
\ = \ \sum_{i+j+l= k} (i+1) 3^{-i} 2^{-j-l-1}
\end{align*}
where the last equalities follow from substituting $u/2$ for $t$ and recalling the integral form \eqref{eq:gamma} of the Gamma function. Hence we conclude
\begin{align*}
I_k - I'_k \ll \sum_{i\le k} (i+1) 3^{-i} 2^{i-k}\sum_{j\le k-i} 1 \le k2^{-k}\sum_{i\le k} (i+1) (2/3)^{i}
\ll k2^{-k}.
\end{align*}

\subsection{Proof of Lemma \ref{lem:I1minus1}}

Recall
\begin{align}\label{eq:ikprimerecall}
I'_k&= \sum_{i=0}^{k}\frac{1}{i!} G_2^{(0,i)}(1,0) \sum_{j+l = k-i} \int_{1}^{2} \frac{\left(-\log(s-1)\right)^l}{l!}2^{-js}  \dd s.
\end{align}
We will prove in the next subsection that
\begin{lemma}\label{lem:tech}
For $k \ge 1$, we have
\begin{align*}
\sum_{j+l=k} \int_{1}^{2} \frac{\left(-\log(s-1)\right)^l}{l!}2^{-sj}  \dd s = 2 -  \frac{\log 2}{4} 2^{-k}\Big(k^2+O\big((k+1)\log (k+2)\big)\Big).
\end{align*}
\end{lemma}
\begin{remark}
The relative saving is $k/\log k$ and it appears sharp. We find it to be an unusual saving.
\end{remark}
Using Lemma~\ref{lem:tech} we simplify the inner sum in \eqref{eq:ikprimerecall} and find
\begin{equation}\label{eq:Ik1B}
    \begin{split}
I'_k &= \sum_{i=0}^{k}\frac{1}{i!} G_2^{(0,i)}(1,0) \bigg(2 -  \frac{\log 2}{4} 2^{i-k}\Big((k-i)^2+O\big((k+1-i)\log (k+2-i)\big)\Big)\bigg) \\
&= 2I'_{k,2} -  \frac{\log 2}{4}2^{-k} B 
\end{split}
\end{equation}
where
\begin{align}
 I'_{k,2} &:=\sum_{i=0}^{k} \frac{1}{i!}G_2^{(0,i)}(1,0),\\
\label{eq:B}B&:=\sum_{i=0}^{k} \frac{2^{i}}{i!}\Big((k-i)^2+ O\big((k-i+1)\log(k-i+2)\big)\Big) G_2^{(0,i)}(1,0).
\end{align}
Lemma~\ref{lem:I1minus1y} with $y=2$ yields
\[ I'_{k,2} =G_2(1,1)+O(3^{-k}) = \tfrac{1}{2} + O(3^{-k}).\]
Similarly as in the proof of Lemma~\ref{lem:I1minus1y}, Cauchy's integral formula and residue theorem imply
\begin{equation}\label{eq:B1}
\begin{split}
B' &:= \sum_{i=0}^{k}\frac{2^i}{i!} G_2^{(0,i)}(1,0) 
=[z^k] \frac{G_2(1,2z)}{1-z} \\
&= \frac{1}{2\pi i} \int_{|z|=1/4} \frac{G_2(1,2z)}{1-z} \frac{\dd z }{z^{k+1}} \\
&=\frac{1}{2\pi i} \int_{|z|=2} \frac{G_2(1,2z)}{1-z} \frac{\dd z }{z^{k+1}} + G_2(1,2) -\frac{\lim_{z \to 3/2} (z-3/2)\,G_y(1,2z)}{(1-3/2)(3/2)^{k+1}} \\
&= d + O((3/2)^{-k}). 
\end{split}
\end{equation}
Here $G_2(1,2)=\tfrac{1}{4}\prod_{p>2}(1-\tfrac{2}{p})^{-1}(1-\tfrac{1}{p})^2 = d$. We also note $G_2(1,z)$ is meromorphic in $|z|<4$ with simple pole at $z=3$ (so $z=3/2$ is the smallest pole of $G_2(1,2z)$).
In particular,
\begin{align}\label{eq:ziG21z}
\frac{1}{i!} G_2^{(0,i)}(1,0) = [z^i]G_2(1,z) \ll 3^{-i}.
\end{align}
Thus combining \eqref{eq:B}, \eqref{eq:B1} and \eqref{eq:ziG21z}, we obtain
\begin{align*}
B-k^2 B'&= \sum_{i=0}^{k} \frac{2^i}{i!} G_2^{(0,i)}(1,0) ((k-i)^2+O((k-i+1)\log(k-i+2))-k^2) \\
& \ll \sum_{i=0}^{ k} (2/3)^i (ik+O(k\log (k+1)))\ll k\log (k+1).
\end{align*}
Hence $B = dk^2 + O(k\log (k+1))$, so plugging back into \eqref{eq:Ik1B} we conclude
\begin{align}
I'_k&= 1 -  \frac{\log 2}{4}2^{-k} \big(dk^2 + O(k\log (k+1))\big).
\end{align}

\subsection{Proof of Lemma \ref{lem:tech}}
We may suppose $k \ge 2$. Multiplying through by $2^k$, we aim to prove 
\begin{align}\label{eq:Aklem}
A_k = 2^{k+1}-\frac{\log 2}{4} k^2 + O(k\log k),
\end{align}
for
\[ A_k := 2^k\sum_{i+j=k} \int_{1}^{2} \frac{\left(-\log(s-1)\right)^{i}}{i!}2^{-js}\dd s = \sum_{i=0}^{k} \frac{2^i}{i!}J(i)\]
where, substituting $s=1+e^{-u}$,
\begin{align*}
J(i)  & := \int_{1}^{2} \left(-\log(s-1)\right)^{i}2^{-(k-i)(s-1)}\dd s
= \int_{0}^{\infty} u^i e^{-u} 2^{-(k-i)e^{-u}}\dd u.
\end{align*}
Note the trivial bound $J(i)\le i!$, using $2^{-(k-i)e^{-u}}\le 1$.

In order to conclude \eqref{eq:Aklem}, it suffices to prove the following two estimates and apply them with  $T = 15\log k$:
\begin{align}
\sum_{0 \le i \le T} \frac{2^i}{i!}J(i) &=  \sum_{0\le i\le T} 2^i + O(kT) \qquad\,\,\,\qquad\qquad\qquad\quad \text{for} \quad k \ge T \ge \log k, \label{eq:smalli}\\
\sum_{T<i \le k} \frac{2^i}{i!}J(i) &= 2^{k+1}-\frac{\log 2}{4} k^2 - \sum_{0 \le i \le T} 2^i + O(kT) \qquad \text{for} \quad  k\ge T \ge 15 \log k. \label{eq:largei}
\end{align}

We first prove \eqref{eq:smalli}. For $T \ge \log k$, the contribution of $e^u \ge k$ to $J(i)$, i.e.~$u \ge \log k$, is handled by the Taylor expansion $2^{-(k-i)e^{-u}}=1-O(ke^{-u})$. Thus
\begin{align*}
J(i) &= \int_{0}^{\log k} u^i e^{-u} 2^{-(k-i)e^{-u}} \dd u + \int_{\log k}^{\infty} u^i e^{-u} (1+O(ke^{-u}))\dd u\\
&= \int_{0}^{\log k} u^i e^{-u} 2^{-(k-i)e^{-u}} \dd u + \int_{\log k}^{\infty} u^i e^{-u} \dd u + O(k 2^{-i}) \int_{0}^{\infty} v^i e^{-v}\dd{v}\\
&= \int_{0}^{\log k} u^i e^{-u} 2^{-(k-i)e^{-u}} \dd u + \int_{0}^{\infty} u^i e^{-u} \dd u - \int_0^{\log k} u^i e^{-u}\dd u + O(k 2^{-i} i!) \\
&=  i!  + \int_{0}^{\log k} u^i e^{-u} (2^{-(k-i)e^{-u}}-1) \dd u +  O(k 2^{-i} i!).
\end{align*}
Summing over $i\le T$, we obtain
\begin{align}\label{eq:smalli1}
\sum_{0\le i\le T}\frac{2^i}{i!}J(i)=\sum_{0\le i\le T}2^i- \sum_{0\le i\le T}\frac{2^i}{i!}\int_{0}^{\log k} u^i e^{-u} (1-2^{-(k-i)e^{-u}}) \dd u    +O( k  T).
\end{align}
Thus to conclude \eqref{eq:smalli} from \eqref{eq:smalli1}, it remains show
\begin{align}\label{eq:smalli2}
\sum_{0\le i\le T}\frac{2^i}{i!}\int_{0}^{\log k} u^i e^{-u} (1-2^{-(k-i)e^{-u}}) \dd u = O(k\log k).
\end{align}
The contribution of $i \le \log k$ is
\begin{align}\label{eq:smalli21}
\sum_{0\le i\le \log k}\frac{2^i}{i!}\int_{0}^{\infty} u^i e^{-u}\dd u \le \sum_{0\le i\le \log k}2^i = O(2^{\log k})=O(k)
\end{align}
since $2<e$. Next recall the function $u \mapsto u^i e^{-u}$ is increasing for $u \le i$, implying
\begin{align}\label{eq:smalli22}
\sum_{\log k<i\le T}\frac{2^i}{i!}\int_{0}^{\log k} u^i e^{-u} \dd u &\le \sum_{\log k<i\le T}\frac{2^i}{i!}\log k (\log k)^i e^{-\log k}  \le \frac{\log k}{k} \sum_{i=0}^{\infty} \frac{(2\log k)^i}{i!} = k \log k.
\end{align}
Combining \eqref{eq:smalli21} and \eqref{eq:smalli22} gives \eqref{eq:smalli2} as desired. This completes the proof of \eqref{eq:smalli}.

Now to prove \eqref{eq:largei}, we begin by handling the contribution of small $u$ to the integral $J(i)$. Since the function $u\mapsto u^i e^{-u}$ increases for $u\le i$, for fixed $a \in (0,1)$, the contribution of $u \le ai$ to $J(i)$ is at most
\begin{align*}
\int_{0}^{ai} u^i e^{-u} 2^{-(k-i)e^{-u}}\dd u\le \int_{0}^{ai} u^i e^{-u}\dd u \le (ai)(ai)^i e^{-ai} \ll_a i!\sqrt{i} (ea/e^a)^i  
\end{align*}
by Stirling's approximation. Thus when $a$ is small enough to satisfy $e^a  > 2ea$, we have
\begin{equation}\label{eq:smallu} 
 \int_{0}^{ai} u^i e^{-u} 2^{-(k-i)e^{-u}}\dd u \ll_a i! 2^{-i}.
\end{equation}
For concreteness, we fix $a=0.21$.  Now let $c:=a/3$. For $T<i\le k$, if $u \ge ci$ we have 
\[(k-i)e^{-u} \le ke^{-ci}< ke^{-cT} \le k^{1-15c} = k^{-.05}\]
since $T \ge 15 \log k$. In particular, we may use the 2nd order Taylor expansion
\[ 2^{-(k-i)e^{-u}} = 1 -\log 2 (k-i)e^{-u} + O(k^2e^{-2u}).\]
Thus by \eqref{eq:smallu} we have for $T< i \le k$,
\begin{equation}
\begin{split}\label{eq:fiAij}
J(i) & = \int_{ci}^{\infty} u^i e^{-u}2^{-(k-i)e^{-u}}\dd u + O(i! 2^{-i}) \\\
& = \int_{ci}^{\infty} u^i e^{-u}\left(  1 -\log 2 \,(k-i)e^{-u} + O\left( k^2e^{-2u}\right) \right) \dd u + O(i! 2^{-i}) \\\
& = A_{i,1}-\log 2 (k-i) A_{i,2}+O(k^2 A_{i,3}) + O(i! 2^{-i})
\end{split}
\end{equation}
where, for $j=1,2,3$,
\begin{align} \label{eq:Aij}
	A_{i,j} := \int_{ci}^{\infty} u^i e^{-ju}\dd u=j^{-i-1}\int_{cij}^{\infty} v^i e^{-v}\dd v= j^{-i-1}\,i!\, (1+ O( 2^{-i})).
\end{align}
In the last equality in \eqref{eq:Aij} we used \eqref{eq:smallu}. Plugging \eqref{eq:Aij} back into \eqref{eq:fiAij} and dividing through by $i!$, we find that
\begin{equation}\label{eq:J(i)Aij}
\begin{split}
\frac{J(i)}{i!} & = (1+ O( 2^{-i}))-\log 2 (k-i) 2^{-i-1} (1+ O( 2^{-i}))+O(k^2 3^{-i-1}+2^{-i}) \\
& = 1-\log 2 (k-i) 2^{-i-1}+O(k^2 3^{-i} + 2^{-i}).
\end{split}
\end{equation}
Summing over $i\in (T,k]$, we conclude that
\begin{align*}
\sum_{T<i \le k} \frac{2^i}{i!}J(i) &= \sum_{T<i \le k}\bigg(2^i - \log 2 (k-i)2^{-1} + O(k^2 (2/3)^{i} + 1)\bigg)\\
&= 2^{k+1} - \sum_{i \le T}2^i - \frac{\log 2}{2} \sum_{T<i \le k} (k-i) \ + \ O(k^2 (2/3)^{T} + k)\\
& =2^{k+1} - \sum_{i \le T}2^i -\frac{\log 2}{2} \Big( \frac{k^2}{2} + O(Tk)\Big)
\end{align*}
holds. Here $k^2(2/3)^{T} \le k^{2+15\log(2/3)} < k^{-4}$ since $T \ge 15\log k$. This gives \eqref{eq:largei} as desired, and hence completes the proof of Lemma~\ref{lem:tech}.

\section{Probability-theoretic argument}

In this section, we give an alternative probabilistic interpretation of Erd\H{o}s sums, showing 

\begin{proposition}\label{prop:fkprob}
We have $f_k=1+O(k/2^{k/4})$.
\end{proposition}

In view of Theorem~\ref{thm:fkthm} we haven't tried to optimize the exponent $2^{k/4}$.

For an integer $a\ge1$, let $P^+(a)$ and $P^-(a)$ denote the largest and smallest prime factors of $a$, respectively (here $P^+(1):=1$ and $P^-(1):=1$). Also let $P_j(n)$ denote the $j$th largest prime of $n$, with multiplicity, so that $n=P_1(n)\cdots P_k(n)$. In particular $P_1(n)=P^+(n)$.

Define the set of L-multiples ${\rm L}_a$,\footnote{L for lexicographic}
\begin{align*}
{\rm L}_a := \{ba\in\N : P^-(b) \ge P^+(a)\}.
\end{align*}
We define the (natural) density of a set $A \subseteq \N$ to be ${\rm d}(A):=\lim_{x \to \infty} |A \cap [1,x]|/x$ as long as this limit exists. Note ${\rm d}({\rm L}_a) = \frac{1}{a}\prod_{p<P^+(a)}(1-1/p)$.

\subsection{Preliminary lemmas}

We begin with some preliminaries.

\begin{lemma}\label{lem:dLareduce}
For any $a\in\N$, we have
\begin{align*}
\sum_{p\ge P^+(a)}{\rm d}({\rm L}_{ap}) = {\rm d}({\rm L}_{a}).
\end{align*}
\end{lemma}
\begin{proof}
Let $y>1$. Consider the set of positive integers without prime factors smaller than $y$, and partition it according to the smallest prime factor $q\ge y$. This gives the disjoint union,
\begin{align*}
\big\{b\in\N : P^+(b)\ge y\big\} \ = \ 
\bigcup_{q\ge y}\big\{bq\in\N :  P^-(b) \ge q\big\}.
\end{align*}
Taking the density of both sides, we find that
\begin{align}\label{eq:iddensity}
\prod_{p<y}\Big(1-\frac{1}{p}\Big) = \sum_{q\ge y}\frac{1}{q}\prod_{p<q}\Big(1-\frac{1}{p}\Big).
\end{align}
Now choosing $y=P^+(a)$, we divide \eqref{eq:iddensity} by $a$ to conclude ${\rm d}({\rm L}_{a})=\sum_{q\ge P^+(a)}{\rm d}({\rm L}_{aq})$.
\end{proof}
From Lemma~\ref{lem:dLareduce}, a simple induction argument on $j\ge1$ gives
\begin{align}\label{eq:dLareduce}
\sum_{\substack{\Omega(b)=j\\p(b)\ge P^+(a)}}{\rm d}({\rm L}_{ab}) = {\rm d}({\rm L}_{a}).
\end{align}
In particular when $a=1$, for any $j\ge1$ we have $\sum_{\Omega(b)=j}{\rm d}({\rm L}_{b})=1$. We shall refine this result in the lemma below.

\begin{lemma}\label{lem:vdLn}
Uniformly for $0<v<1$ and $a\in\Z_{>1}$, we have
\begin{align}
\sum_{q\ge P_1(a)^{\frac{1}{v}}}{\rm d}({\rm L}_{aq}) = v\,{\rm d}({\rm L}_a)\Big(1+O\Big(
\tfrac{1}{\log P_1(a)}\Big)\Big).
\end{align}
\end{lemma}\begin{proof}
Take $0<v<1$. We first recall Mertens' product theorem states that
\begin{align*}
\prod_{p< x}\Big(1-\frac{1}{p}\Big) = \frac{e^{-\gamma}}{\log x}\Big(1 + O\Big(\frac{1}{\log x}\Big)\Big)
\end{align*}
holds for $x \ge 2$. In particular, for $x=P_1(a) \ge 2$,
\begin{equation}\label{eq:PaPav}
\begin{split}
\prod_{P_1(a)\le p< P_1(a)^{\frac{1}{v}}}\Big(1-\frac{1}{p}\Big) & = \prod_{p<P_1(a)}\Big(1-\frac{1}{p}\Big)^{-1}\prod_{p< P_1(a)^{\frac{1}{v}}}\Big(1-\frac{1}{p}\Big) \\
& = \frac{\log P_1(a)}{\log P_1(a)^{\frac{1}{v}}}\Big(1 + O\Big(\frac{1}{\log P_1(a)}\Big)\Big)
\ = \ v\Big(1+O(1/\log P_1(a))\Big).
\end{split}
\end{equation}

So by \eqref{eq:PaPav} and \eqref{eq:iddensity} with $y=P_1(n)^{\frac{1}{v}}$,
\begin{align*}
\sum_{q\ge P_1(a)^{\frac{1}{v}}}\frac{1}{q}\prod_{p<q}\Big(1-\frac{1}{p}\Big) = \prod_{p<P_1(a)^{\frac{1}{v}}}\Big(1-\frac{1}{p}\Big) = v\Big(1+O(1/\log P_1(a))\Big)\prod_{p<P_1(a)}\Big(1-\frac{1}{p}\Big).
\end{align*}
Dividing through by $a$ completes the proof.
\end{proof}

\begin{lemma}\label{lem:mass}
For $k\ge1$, let $c_1\ge\cdots\ge c_k\ge 0$. If $d_1,D_1,E_1,\ldots,d_k,D_k,E_k\ge0$ satisfy $E_i\le \sum_{j=1}^i d_j \le D_i$ for all $1\le i\le k$ (and let $d_0=E_0=D_0=0$), then we have
\begin{align*}
\sum_{i=1}^k c_i(E_i - E_{i-1}) \ \le \
\sum_{i=1}^k c_i d_i \ \le \ \sum_{i=1}^k c_i(D_i - D_{i-1}).
\end{align*}
\end{lemma}
\begin{proof}
We have
\begin{align}\label{eq:monotoneid}
\sum_{i=1}^k c_id_i 
= \sum_{i=1}^kc_i\Big(\sum_{j=1}^id_j-\sum_{j=0}^{i-1}d_j\Big)= \sum_{i=1}^{k-1}(c_i-c_{i+1})\sum_{j=1}^{i}d_j \ + \ c_k\sum_{i=1}^{k}d_i
\end{align}
by summation by parts. Since $c_i - c_{i+1} \ge0$ and $\sum_{j\le i} d_j \le D_i$, from \eqref{eq:monotoneid} we obtain that
\begin{align*}
\sum_{i=1}^k c_id_i  \le \sum_{i=1}^{k-1} (c_i-c_{i+1})D_i \ + \ c_kD_k = \sum_{i=1}^{k} c_i (D_i-D_{i-1})
\end{align*}
holds. Similarly, since $\sum_{j\le i} d_j \ge E_i \ge0$, from \eqref{eq:monotoneid} we obtain that
\begin{align*}
\sum_{i=1}^{k} c_id_i \ge \sum_{i=1}^{k-1} (c_i-c_{i+1})E_i \ + \ c_kE_k = \sum_{i=1}^{k} c_i (E_i-E_{i-1})
\end{align*}
holds.
\end{proof}

To handle the contribution of smooth numbers, we use a simple bound of Erd\H{o}s and S\'ark\"ozy \cite[Lemma~2]{ErdSark}, whose proof we provide for completeness.

\begin{lemma}[Erd\H{o}s--S\'ark\"ozy] \label{lem:ErdSark}
For any $k\ge1$, $y>1$, we have
\begin{align*}
\sum_{\substack{\Omega(n)=k\\P_1(n)<e^y}}\frac{1}{n} \ \ll \ y^2 \,k/2^{k}.
\end{align*}
\end{lemma}
\begin{proof}
Observe that $2^k$ times our given sum is bounded by the following Euler product,
\begin{align*}
\sum_{\substack{\Omega(n)=k\\P_1(n)<e^y}}\frac{2^k}{n} & \le \prod_{p<e^y}\Big(1 + \frac{2}{p}+\cdots+\frac{2^k}{p^k}\Big)\\
& \le (k+1)\prod_{2<p<e^y}\Big(1 - \frac{2}{p}\Big)^{-1} \ll ky^2
\end{align*}
by Mertens' product theorem. Dividing by $2^k$ completes the proof.
\end{proof}

\begin{corollary}\label{cor:ksmooth}
For any $1\le j\le k$ and $y>1$, we have
\begin{align*}
\sum_{\substack{\Omega(n)=k\\P_{j+1}(n)<e^{y}}} \frac{1}{n \log n} \ll \sum_{\substack{\Omega(n)=k\\P_{j+1}(n)<e^{y}}}{\rm d}({\rm L}_n) \ll y^2\, k\,2^{j-k}.
\end{align*}
\end{corollary}
\begin{proof}
First, for each $n$ with $\Omega(n)=k$, one can factor $n$ uniquely as $ab$ with $\Omega(b)=j$ and $p(b)\ge P_1(a)$ (namely take $b= \prod_{i=1}^{j}P_i(n)$ and $a=n/b$). Thus by \eqref{eq:dLareduce} we have
\begin{align*}
\sum_{\substack{\Omega(n)=k\\P_{j+1}(n)<e^{y}}}{\rm d}({\rm L}_n) = \sum_{\substack{\Omega(a)=k-j\\P_1(a)<e^{y}}} \sum_{\substack{\Omega(b)=j\\p(b)\ge P_1(a)}}{\rm d}({\rm L}_{ab})= \sum_{\substack{\Omega(a)=k-j\\P_1(a)<e^{y}}}{\rm d}({\rm L}_a).
\end{align*}
On the right-hand side of the above identity we apply the simple bound  ${\rm d}({\rm L}_a) \ll 1/a$, and on the left-hand side we apply ${\rm d}({\rm L}_n) \gg 1/(n \log n)$. This gives
\begin{align*}
\sum_{\substack{\Omega(n)=k\\P_{j+1}(n)<e^{y}}} \frac{1}{n \log n}\ll \sum_{\substack{\Omega(n)=k\\P_{j+1}(n)<e^{y}}}{\rm d}({\rm L}_n) \ll \sum_{\substack{\Omega(a)=k-j\\P_1(a)<e^{y}}}\frac{1}{a} \ll y^2\, k\,2^{j-k}
\end{align*}
by Lemma~\ref{lem:ErdSark} with $k$ replaced by $k-j$.
\end{proof}

\subsection{Poof of Proposition \ref{prop:fkprob}}

Let $k\ge1$ be sufficiently large. We shall choose $y=2^j$ for $j=\lfloor k/4\rfloor$, and $N=4^k$. Let $f'_k$ denote the sum $f_k$ restricted by $P_{j+1}(n)\ge e^{y}$. Thus by Corollary \ref{cor:ksmooth},
\begin{align}\label{eq:fkfk1}
f_k=\sum_{\Omega(n)=k}\frac{1}{n\log n} = f_k' + O(y^2 k 2^{j-k}) = f_k' + O(k/2^{k/4})
\end{align}
where, by Mertens' product theorem,
\begin{align} \label{eq:fk1uj0}
f_k' &:= \sum_{\substack{\Omega(n)=k\\P_{j+1}(n)\ge  e^{y}}}\frac{1}{n\log n}
= \sum_{\substack{\Omega(n)=k\\P_{j+1}(n)\ge e^{y}}}\big(e^\gamma + \tfrac{O(1)}{\log P_1(n)}\big)\frac{\log P_1(n)}{\log n}{\rm d}({\rm L}_n).
\end{align}

Next, we rewrite the identity $n=P_1(n)\cdots P_k(n)$ as
\begin{align*}
\frac{\log n}{\log P_1(n)} &=  1 + \frac{\log P_2(n)}{\log P_1(n)}\Big(1+\cdots\Big(1+\frac{\log P_{j+1}(n)}{\log P_j(n)} \Big(1+\frac{\log\big(P_{j+2}(n)\cdots P_k(n)\big)}{\log P_{j+1}(n)}\Big)\Big)\cdots\Big).
\end{align*}
Taking the reciprocal of identity above gives
\begin{align*}
\frac{\log P_1(n)}{\log n} & = u_{j+1}\Big(\tfrac{\log P_2(n)}{\log P_1(n)},\ldots,\tfrac{\log P_{j+1}(n)}{\log P_j(n)},\tfrac{\log(P_{j+2}(n)\cdots P_k(n))}{\log P_{j+1}(n)}\Big),
\end{align*}
for the functions $u_j
\colon \R^j\to\R$ given by
\begin{align}\label{eq:ukdef}
u_j(x_1,\ldots,x_j) := \frac{1}{1+x_1(1+x_2(\cdots(1+x_j)\cdots))}.
\end{align}
In particular, from $P_1(n)\ge\ldots \ge P_k(n)$ we infer the inequalities
\begin{align}
\frac{\log P_1(n)}{\log n} & \le u_j\Big(\tfrac{\log P_2(n)}{\log P_1(n)},\ldots,\tfrac{\log P_{j+1}(n)}{\log P_j(n)}\Big), \label{eq:ujupper}\\
\frac{\log P_1(n)}{\log n} & \ge u_j\Big(\tfrac{\log P_2(n)}{\log P_1(n)},\ldots,\tfrac{\log P_{j+1}(n)}{\log P_j(n)}(k-j)\Big). \label{eq:lowerupper}
\end{align}
By \eqref{eq:ujupper}, we see \eqref{eq:fk1uj0} implies that
\begin{align} \label{eq:fk1uj}
f_k' \le \sum_{\substack{\Omega(a)=k-j\\P_1(a)\ge e^{y}}}\big(e^\gamma + \tfrac{O(1)}{\log P_1(a)}\big)\sum_{P_1(a) \le p_j\le \cdots\le  p_1}
u_j\Big(\tfrac{\log p_2}{\log p_1},\ldots,\tfrac{\log P_1(a)}{\log p_j}\Big){\rm d}({\rm L}_{ap_j\cdots p_1}).
\end{align}

\begin{lemma}\label{lem:inductmonotone}
There is an absolute constant $C>1$ such that for any $a\in \Z_{>1}$,
\begin{align*}
\sum_{P_1(a) \le p_j\le \cdots\le  p_1}
u_j\Big(\tfrac{\log p_2}{\log p_1},\ldots,\tfrac{\log P_1(a)}{\log p_j}\Big){\rm d}({\rm L}_{ap_j\cdots p_1}) 
\le \frac{\big(1 + \tfrac{C}{\log P_1(a)}\big)^j}{N^j}\sum_{i_1,\ldots, i_j =1}^N u_j\Big(\tfrac{i_1-1}{N},\ldots,\tfrac{i_j-1}{N}\Big){\rm d}({\rm L}_{a}).
\end{align*}
\end{lemma}
\begin{proof}
For each $1\le r\le j$, it suffices to show that
\begin{equation}\label{eq:inductmono}
\begin{split}
&N^{1-r}\sum_{i_1,\ldots i_{r-1}=1}^N\sum_{P_1(a) \le p_j\le \cdots \le p_{r+1}\le  p_r} u_j\Big(\tfrac{i_1-1}{N},\ldots,\tfrac{i_{r-1}-1}{N},\tfrac{\log p_{r+1}}{\log p_r},\ldots,\tfrac{\log P_1(a)}{\log p_j}\Big)\, {\rm d}({\rm L}_{ap_j\cdots p_r}) \\
&\le N^{-r}\sum_{i_1,\ldots,i_r=1}^N\sum_{P_1(a) \le p_j\le \cdots\le  p_{r+1}}  u_j\Big(\tfrac{i_1-1}{N},\ldots,\tfrac{i_{r}-1}{N},\tfrac{\log p_{r+2}}{\log p_{r+1}},\ldots,\tfrac{\log P_1(a)}{\log p_j}\Big)\, {\rm d}({\rm L}_{ap_j\cdots p_{r+1}}) \Big(1 + \tfrac{C}{\log P_1(a)}\Big)
\end{split}
\end{equation}
holds. Indeed, iterating \eqref{eq:inductmono} (with each $r=1,2,\ldots, j$ in turn) completes the proof of the lemma.

To show that \eqref{eq:inductmono} holds, fix indices $i_1,\ldots i_{r-1}\le N$ and primes $p_j\le \cdots\le  p_{r+1}$ ($p_j \ge P_1(a)$). Define $c_{i_r}$ and $d_{i_r}$ by
\begin{align*}
c_{i_r} &:= u_j\Big(\tfrac{i_1-1}{N},\ldots,\tfrac{i_{r}-1}{N},\tfrac{\log p_{r+2}}{\log p_{r+1}},\ldots,\tfrac{\log P_1(a)}{\log p_j}\Big)\\
d_{i_r} &:= \sum_{p_r\in [p_{r+1}^{N/i_r}, p_{r+1}^{N/(i_r-1)})} {\rm d}({\rm L}_{ap_j\cdots p_r}).
\end{align*}
(For $i_r=1$, the range of $p_r$ in the definition of $d_{i_r}$ is to be interpreted as $[p_{r+1}^{N},\infty)$.)
Note for any $u\le N$, by Lemma~\ref{lem:vdLn} we have
\begin{align*}
\sum_{i_r=1}^u d_{i_r} = \sum_{p_r \ge p_{r+1}^{N/u}}{\rm d}({\rm L}_{ap_j\cdots p_r}) 
\le \frac{u}{N}{\rm d}({\rm L}_{ap_j\cdots p_{r+1}})\Big(1 + \tfrac{C}{\log P_1(a)}\Big) =: D_u.
\end{align*}
In particular $D_{u}-D_{u-1} = \frac{1}{N}{\rm d}({\rm L}_{ap_j\cdots p_{r+1}})(1+C/\log P_1(a))$. Splitting up the sum over $p_r\ge p_{r+1}$ below according to the $i_r$ for which $p_r\in [p_{r+1}^{N/i_r}, p_{r+1}^{N/(i_{r}-1)})$ holds, and then applying Lemma \ref{lem:mass}, we find that
\begin{equation}\label{eq:inductmonotone}
\begin{split}
\sum_{p_r\ge p_{r+1}} & u_j\Big(\tfrac{i_1-1}{N},\ldots,\tfrac{i_{r-1}-1}{N},\tfrac{\log p_{r+1}}{\log p_r},\ldots,\tfrac{\log P_1(a)}{\log p_j}\Big)\, {\rm d}({\rm L}_{ap_j\cdots p_r}) \\
& \le \sum_{i_r=1}^Nc_{i_r}\sum_{p_r\in [p_{r+1}^{N/i_r}, p_{r+1}^{N/(i_{r}-1)})}  {\rm d}({\rm L}_{ap_j\cdots p_r}) = \sum_{i_r=1}^N c_{i_r}d_{i_r} \\
&\le \sum_{i_r=1}^N c_{i_r}(D_{i_r}-D_{i_r-1})
= \Big(1 + \tfrac{C}{\log P_1(a)}\Big)\frac{1}{N}\sum_{i_r=1}^N c_{i_r}{\rm d}({\rm L}_{ap_j\cdots p_{r+1}}) \\
& \ = \Big(1 + \tfrac{C}{\log P_1(a)}\Big)\frac{1}{N}\sum_{i_r=1}^N u_j\Big(\tfrac{i_1-1}{N},\ldots,\tfrac{i_{r}-1}{N},\tfrac{\log p_{r+2}}{\log p_{r+1}},\ldots,\tfrac{\log P_1(a)}{\log p_j}\Big)\, {\rm d}({\rm L}_{ap_j\cdots p_{r+1}}). 
\end{split}
\end{equation}
Summing \eqref{eq:inductmonotone} over $i_1,\ldots, i_{r-1}\le N$ and $p_j\le \cdots\le  p_{r+1}$, we obtain \eqref{eq:inductmono} as desired.
\end{proof}

Plugging Lemma~\ref{lem:inductmonotone} into \eqref{eq:fk1uj} we obtain that
\begin{align}
f_k' &\le e^\gamma\sum_{\substack{\Omega(a)=k-j\\P_1(a)\ge e^{y}}}{\rm d}({\rm L}_{a})\,\frac{\big(1 + C/y\big)^{j+1}}{N^j}\sum_{i_1,\ldots, i_j =1}^N u_j\Big(\tfrac{i_1-1}{N},\ldots,\tfrac{i_j-1}{N}\Big)
\end{align}
holds. By Corollary \ref{cor:ksmooth} and \eqref{eq:dLareduce} with $a=1$,
\begin{align*}
\sum_{\substack{\Omega(a)=k-j\\P_1(a)\ge e^{y}}}{\rm d}({\rm L}_a) = \sum_{\substack{\Omega(a)=k-j}}{\rm d}({\rm L}_a) + O(y^2 k 2^{j-k}) = 1 +O(k/2^{k/4}),
\end{align*}
recalling the definitions $y=2^j$, $j=\lfloor k/4\rfloor$. Thus
\begin{align}\label{eq:f1kleIj0}
f_k' &\le e^\gamma\,\frac{\big(1 + O(k/2^{k/4})\big)^{j+1}}{N^j}\sum_{i_1,\ldots, i_j =1}^N u_j\Big(\tfrac{i_1-1}{N},\ldots,\tfrac{i_j-1}{N}\Big).
\end{align}

By an analogous argument (using the lower bound in \eqref{eq:lowerupper}, and $E_i = \frac{i}{N}\, {\rm d}({\rm L}_{p_2\cdots p_ja}) (1 - C/\log P_1(a))$ in Lemma~\ref{lem:mass} instead of $D_i$), we may obtain a similar lower bound
\begin{align}\label{eq:f1kgeIj0}
f_k' &\ge e^\gamma\frac{\big(1 - O(k/2^{k/4}))\big)^{j+2}}{N^j}\sum_{i_1,\ldots, i_j =1}^N u_j\Big(\tfrac{i_1}{N},\ldots,\tfrac{i_j}{N}(k-j)\Big).
\end{align}

Now for a sequence $(c_j)_j$, define the integral
\begin{align}\label{eq:Ijcj}
I_j(c_j) := \int_{[0,1]^j} u_j(x_1,\ldots,x_{j-1},c_jx_j)\, \dd x_1\cdots \dd x_j.
\end{align}

Observe that the sum in \eqref{eq:f1kleIj0} is the upper Riemann sum for the integral $I_j(1)$, noting that $u_j:[0,1]^j\to[0,1]$ is decreasing in each component. And since the upper and lower Riemann sums (which squeeze $I_j(1)$) overlap in $(N-1)^j$ points, their difference is $\ll (N^j - (N-1)^j) / N^j = 1-(1-1/N)^j \ll j/N$. In particular the sum in \eqref{eq:f1kleIj0} equals $I_j(1) + O(j/N)$.

Similarly \eqref{eq:f1kleIj0} is the lower Riemann sum for $I_j(k-j)$. Thus we obtain
\begin{align*}
N^{-j}\sum_{i_1,\ldots, i_j =1}^N u_j\Big(\tfrac{i_1-1}{N},\ldots,\tfrac{i_j-1}{N}\Big) &= I_j(1) + O(j/N),\\
N^{-j}\sum_{i_1,\ldots, i_j =1}^N u_j\Big(\tfrac{i_1}{N},\ldots,\tfrac{i_j}{N}(k-j)\Big) &= I_j(k-j) + O(j/N).
\end{align*}
Recalling $N=4^k$ and $j=\lfloor k/4\rfloor$, we see that \eqref{eq:f1kleIj0} and \eqref{eq:f1kgeIj0} become
\begin{align}
f'_k &\le (e^{\gamma}+O(k/2^{k/4}))\,I_j(1), \label{eq:f1kleIj}\\
f'_k &\ge (e^{\gamma}-O(k/2^{k/4}))\, I_j(k-j).\label{eq:f1kgeIj}
\end{align}

In the next section, we shall establish the following quantitative result.
\begin{theorem}\label{theo:problem1}
Let $(c_j)_j$ be any nonnegative sequence. Then $I_j(c_j)$, as in \eqref{eq:Ijcj}, satisfies
\begin{align*}
    I_j(c_j) = e^{-\gamma} + O\left(2^{-j}\left(1 + c_j\right)\right).
\end{align*}
\end{theorem}
\begin{remark}
The qualitative result that $I_n(c_n)=e^{-\gamma}+o(1)$ may be established in the wider regime where $\limsup_{n \to \infty} \frac{1}{n} \log c_n < 1$ holds (see Lemma~\ref{lem:I_error} for more precise statement) but the bound above is sufficient for the purpose of this article. 
\end{remark}

In particular, Theorem~\ref{theo:problem1} gives
\begin{align*}
I_j(1)  &= e^{-\gamma}+O(2^{-j}),\\
I_j(k-j) &=e^{-\gamma}+O(k/2^{j}).
\end{align*}
Thus plugging into \eqref{eq:f1kleIj} and \eqref{eq:f1kgeIj} we obtain $f'_k= 1 +O(k/2^{k/4})$. Hence by \eqref{eq:fkfk1} we conclude that
\begin{align}
f_k = f_k'+O(k/2^{k/4}) = 1 + O(k/2^{k/4}).
\end{align}
This completes the proof of Proposition~\ref{prop:fkprob}.

\section{A sequence of integrals}

In this section, we prove Theorem~\ref{theo:problem1}. This implies Theorem~\ref{thm:iteratedIk} for $I_j(1)$. Recalling $u_j$ in \eqref{eq:ukdef}, we defined the following sequence of integrals, for a sequence $(c_j)_j$,
\begin{align}\label{def:Incn}
I_j = I_j(c_j) := \int_{[0,1]^j} \frac{\dd x_1 \dd x_2 \cdots \dd x_j}{1+x_1(1+x_2(\cdots(1+x_{j-1}(1+c_jx_j))\cdots))}.
\end{align}

This sequence of iterated integrals is closely related to the so-called Dickman--Goncharov distribution, the properties of which are well studied in the literature (see e.g.~\cite[Props.~2.1 and 2.4]{Molchanov},  \cite{Chamayou} and \cite{PW2004}). Since we need small refinements of existing results, we will provide self-contained explanations for all the results below. Our approach will be based on techniques from random iterated functions/stochastic fixed-point equations.

\subsection{Probabilistic setup}
In the following, all random variables are assumed to live in a common reference probability space $(\Omega, \mathcal{F}, \mathbb{P})$.

\begin{lemma}\label{lem:iteratedmap}
Let $U, U_1, U_2, \dots$ be i.i.d. $\mathrm{Uniform}[0,1]$ random variables. Define
\begin{align*}
F_n(x) := 1 + U_n x \qquad \forall x \in \R, \quad n \in \N,
\end{align*}
and consider the sequence of iterated random functions
\begin{align}\label{eq:Sn}
S_0(x):= x, \qquad S_n(x) := F_1 \circ F_2 \circ \dots \circ F_n(x) \qquad \forall n \in \N.
\end{align}
Then the following statements hold.
\begin{itemize}
\item[(i)] The random variable
\begin{align}\label{eq:Sinfty}
S_\infty := \lim_{n \to \infty} S_n(1) = 1 + \sum_{j=1}^\infty \prod_{k=1}^j U_k
\end{align}

\noindent exists almost surely and satisfies $\mathbb{P}(1 \le S_\infty < \infty) = 1$.
\item[(ii)] Let $\theta \in (1, e)$. If $(V_n)_n$ is a sequence of random variables such that  $\lim_{n \to \infty} \theta^{-n} |V_n| = 0$ almost surely, then
\begin{align}\label{eq:S_conv}
\lim_{n \to \infty} S_n(V_n) = S_\infty \qquad \text{almost surely}.
\end{align}
\end{itemize}
\end{lemma}

\begin{remark}
The composition of maps $S_n(x) := F_1 \circ \cdots \circ F_n(x)$ in \Cref{lem:iteratedmap} may be identified with products of random matrices, i.e. 
\begin{align*}
\begin{pmatrix} S_n(x)  \\ 1\end{pmatrix}
= 
\begin{pmatrix} U_1 & 1 \\ 0 & 1 \end{pmatrix} \cdots \begin{pmatrix} U_n & 1 \\ 0 & 1 \end{pmatrix}
\begin{pmatrix} x \\ 1 \end{pmatrix} \qquad \forall x \in \mathbb{R}, \quad n \in \mathbb{N}.
\end{align*}

\noindent We choose the current formulation because many results in this section have natural extensions to nonlinear random functions $F_n$ with similar assumptions on their Lipschitz constants.
\end{remark}

\begin{proof}
By definition,
\begin{align*}
S_n(1) = 1 + \sum_{j=1}^n \prod_{k=1}^j U_k
\end{align*}
and it is immediate that $S_{n+1}(1) \ge S_n(1) \ge 1$ for all $n \in \N$. Therefore, the almost sure limit  in \eqref{eq:Sinfty} exists by monotone convergence and we have $\mathbb{P}(S_\infty \ge 1) = 1$. Moreover, 
\begin{align}\label{eq:ESinfty}
\E[S_\infty] 
= 1 + \sum_{j=1}^\infty \E\left[\prod_{k=1}^j U_k\right]
= 1 + \sum_{j=1}^\infty \E[U]^j
=\sum_{j=0}^\infty 2^{-j}
= 2
\end{align}
which implies $\mathbb{P}(S_\infty < \infty) = 1$. Thus we have verified (i).

Now suppose $\theta \in (1, e)$, and $(V_n)_n$ is a sequence of random variables such that $\theta^{-n} |V_n|$ converges almost surely to $0$ as $n \to \infty$. Since each of the functions $F_n$ is linear with Lipschitz constant $\|F_n\|_{\mathrm{Lip}} = U_n$, we have
\begin{equation}\label{eq:uniformbd}
\begin{split}
|S_n(V_n) -S_n(1)|
&= \left| F_1 \circ F_2 \circ \dots \circ F_n(V_n)  - F_1 \circ F_2 \circ \dots \circ F_n(1)\right|  \\
& = \left[\prod_{j=1}^n 
\|F_j\|_{\mathrm{Lip}}\right]  |V_n-1| = \left[\prod_{j=1}^n 
U_n\right] |V_n-1| \\
& \le \exp\left(\sum_{j=1}^n \log U_j\right) (1 + |V_n|).
\end{split}
\end{equation}
Since $\E[\log U_i] = \int_0^1 \log u\,\dd{u} = -1$, the strong law of large numbers gives
\begin{align*}
\frac{1}{n} \sum_{j=1}^n \log U_j \xrightarrow[n \to \infty]{a.s.} -1.
\end{align*}
In particular, if we choose $\varepsilon \in (0, 1 - \log \theta)$, then almost surely there exists some (random) $n_0 = n_0(\varepsilon) \in \N$ such that
\begin{align*}
    \frac{1}{n} \sum_{j=1}^n \log U_j \le -1+\varepsilon
    \qquad \text{for all $n \ge n_0$.}
\end{align*}
Substituting this into \eqref{eq:uniformbd}, we obtain
\begin{align*}
|S_n(V_n) -S_n(1)| & \le (\theta/e^{1-\varepsilon})^n \left[\theta^{-n} (1 + |V_n|)\right]
= O\left(\theta^{-n} (1 + |V_n|)\right).
\end{align*}
The assumption on $V_n$ implies that $S_n(V_n)-S_n(1)$ converges to 0 almost surely. Since $S_n(1)$ converges to $S_\infty$ almost surely by (i), we conclude that (ii) holds, i.e. $S_n(V_n)$ also converges to $S_\infty$ almost surely.
\end{proof}

\begin{corollary}
For the integral $I_n$ as in \eqref{def:Incn}, we have
\begin{align*}
I_n = \E\left[\frac{1}{S_{n-1}(1+c_n U_n )}\right].
\end{align*}
In particular, for any sequence $(c_n)_n$ satisfying $0 \le c_n = o(\theta^n)$ for some $\theta \in (1, e)$,
\begin{align*}
    \lim_{n \to \infty} I_n = \E[S_\infty^{-1}] =: I_\infty.
\end{align*}
\end{corollary}

\begin{proof}
The probabilistic representation of the iterated integral $I_n$ follows immediately by construction. If we let $V_{n-1} := 1+c_n U_n$, then $\theta^{-n} V_n$ converges to $0$ almost surely as $n \to \infty$, and by \Cref{lem:iteratedmap} we also obtain that $S_{n-1}(V_{n-1})$ converges almost surely to $S_\infty$. Since
\begin{align*}
\frac{1}{S_{n-1}(V_{n-1})} \le \frac{1}{S_{n-1}(1)} \le 1,
\end{align*}
we conclude by dominated convergence that
\begin{align*}
\lim_{n \to \infty} I_n=\lim_{n \to \infty} \E\left[\frac{1}{S_{n-1}(V_{n-1})}\right] = \E[S_\infty^{-1}].
\end{align*}

\end{proof}

The next step is a simple but crucial characterisation of the distribution of $S_\infty$.
\begin{lemma}\label{lem:randrecursion}
Let $U,X$ be two independent random variables such that $U \sim \mathrm{Uniform}[0, 1]$ and
\begin{align}\label{eq:recursion}
X  \overset{(d)}{=} 1 + UX.
\end{align}
If $\mathbb{P}(|X| < \infty) = 1$, then $X \overset{(d)}{=} S_\infty$.
\end{lemma}

\begin{remark}\label{rem:indep-probspace}
While \Cref{lem:randrecursion} is conveniently formulated in terms of random variables, the statement ultimately concerns the law of $X$ only and does not require the knowledge of the underlying probability space. For instance, the distributional equality \eqref{eq:recursion} can be reformulated as
\begin{align*}
\mathbb{E}[g(X)] = \int_0^1 \mathbb{E}[g(1+uX)] du
\end{align*}
for all suitable test functions $g$ (provided both sides are well defined), and the conclusion of the lemma says that we necessarily have $\mathbb{E}[g(X)] = \mathbb{E}[g(S_\infty)]$, or equivalently $\mathbb{P}(X \le x) = \mathbb{P}(S_\infty \le x)$ for any $x \in \mathbb{R}$.
\end{remark}

\begin{proof}
Without loss of generality (by \Cref{rem:indep-probspace}), assume that $X$ is defined on the same probability space in \Cref{lem:iteratedmap} such that $X$ is independent of all the uniformly distributed random variables $U, U_1, U_2, U_3, \dots$; our goal is to verify the following two claims:
\begin{enumerate}
\item The random variable $S_\infty := \lim_{n \to \infty} S_n(1)$ satisfies the distributional equality \eqref{eq:recursion}, i.e. $S_\infty \overset{(d)}{=} 1 + U S_\infty$.

This is straightforward by a quick reordering of the underlying i.i.d. random variables/iterated maps. Indeed,
\begin{align*}
S_n(1) = F_1 \circ \dots \circ F_n(1)
&\overset{(d)}{=} F_n \circ F_1 \circ F_2 \circ \dots \circ F_{n-1}(1)\\
& =  1 + U_n S_{n-1}(1)
\xrightarrow[n \to \infty]{(d)} 1 + U S_\infty.
\end{align*}

\item If $X$ satisfies $\mathbb{P}(|X| < \infty)$ and \eqref{eq:recursion}, then $X \overset{(d)}{=} S_\infty$.

To establish this claim, observe that for any $n \in \N$ we have
\begin{align*}
X 
\overset{(d)}{=} F_1(X)
\overset{(d)}{=} \cdots
\overset{(d)}{=} F_1\circ F_2 \circ \dots \circ F_n (X)
=  S_n(X).
\end{align*}
Since $\mathbb{P}(|X| < \infty)$ and in particular $2^{-n} |X| \xrightarrow[n \to \infty]{a.s.} 0$, we apply \Cref{lem:iteratedmap} with $V_n=X$ and obtain $S_n(X) \xrightarrow[n \to \infty]{a.s.} S_\infty$. In other words, 
\begin{align*}
X \overset{(d)}{=} S_\infty = 1 + \sum_{j =1}^{\infty} \prod_{k =1}^j U_k,
\end{align*}
which concludes the proof.
\end{enumerate}
\end{proof}

\subsection{Proof of \Cref{theo:problem1}}
The recursive distributional equation \eqref{eq:recursion} is a very convenient tool that helps us control the rate of convergence of $S_n(\cdot)$ and extract information about the statistical behaviour of $S_\infty$ at the same time. We first explain how to estimate the difference between $I_n(c_n)$ and its limit $I_\infty$.

\begin{lemma}\label{lem:I_error}
For any nonnegative sequence $(c_n)_n$, we have
\begin{align} \label{eq:I_error}
|I_n(c_n) - I_\infty| &\le  2^{1-n} + \E \left[\min\left(1, c_n \prod_{j \le n} U_j\right) \right] \ \le \ 2^{-n} (2+c_n).
\end{align}
\end{lemma}

\begin{proof}
Recall that $V_{n-1} := 1+c_{n} U_{n}$ and
\begin{align*}
I_n(c_n) = \E\left[ \frac{1}{S_{n-1}(V_{n-1})}\right].
\end{align*}
On the other hand, if we introduce a new random variable $T \overset{(d)}{=} S_\infty$ that is independent of all of the $U_j$'s, we see that
\begin{align*}
S_{n-1}(T) \overset{(d)}{=} S_\infty \qquad \text{and hence} \qquad 
I_\infty = \E[S_\infty^{-1}] = \E\left[\frac{1}{S_{n-1}(T)}\right]
\end{align*}
by the distributional fixed point equation \eqref{eq:recursion}. 

Since $S_{n-1}(\cdot)$ is linear with Lipschitz constant $\|S_{n-1}\|_{\mathrm{Lip}} = \prod_{j \le n-1} U_j$, we have
\begin{equation}\label{eq:IncnIinfty}
\begin{split}
|I_n(c_n) - I_\infty|
&= \left|\E\left[\frac{1}{S_{n-1}(V_{n-1})} - \frac{1}{S_{n-1}(T)}\right]\right| \\
& \le \E\left[\frac{\|S_{n-1}\|_{\mathrm{Lip}}\left| V_{n-1} - T\right|}{S_{n-1}(V_{n-1}) S_{n-1}(T)}\right] \\
& \le \E\left[\frac{\|S_{n-1}\|_{\mathrm{Lip}}\left| T - 1\right|}{ S_{n-1}(T)}\right]
+
\E\left[\frac{\|S_{n-1}\|_{\mathrm{Lip}}\left| V_{n-1} - 1\right|}{S_{n-1}(V_{n-1})}\right].
\end{split}
\end{equation}
Since $\E[U_j]=1/2$, $\E[T]=\E[S_\infty]=2$ by \eqref{eq:ESinfty} and $\mathbb{P}(T \ge 1) = \mathbb{P}(S_{n-1}(T) \ge 1) = \mathbb{P}(S_\infty \ge 1) = 1$, the first term on the right-hand side of \eqref{eq:IncnIinfty} satisfies
\begin{align*}
\E\left[\frac{\|S_{n-1}\|_{\mathrm{Lip}}\left| T - 1\right|}{ S_{n-1}(T)}\right] 
& \le \E\left[\|S_{n-1}\|_{\mathrm{Lip}}\left| T - 1\right|\right] \\
& =\E\left[ \left(\prod_{j\le n-1} U_j\right)\left| T - 1\right|\right] = \E[T - 1]\prod_{j \le n-1}\E\left[ U_j\right] = 2^{1-n}.
\end{align*}

Next, observe that $V_{n-1}-1=c_nU_n$ and $S_{n-1}(x) \ge 1$ for any $x \ge 0$. This means
\begin{align*}
\frac{\|S_{n-1}\|_{\mathrm{Lip}}\left| V_{n-1} - 1\right|}{S_{n-1}(V_{n-1})} \le c_n\prod_{j \le n} U_j.
\end{align*}
On the other hand, $S_{n-1}(V_{n-1})\ge  1 +  V_{n-1}\|S_{n-1}\|_{\mathrm{Lip}}\ge \|S_{n-1}\|_{\mathrm{Lip}}\left| V_{n-1} - 1\right|$, which leads to a slightly improved bound
\begin{align*}
\frac{\|S_{n-1}\|_{\mathrm{Lip}}\left| V_{n-1} - 1\right|}{S_{n-1}(V_{n-1})} \le 
\min\left(1, c_n\prod_{j \le n} U_j\right).
\end{align*}
Taking expectation both sides and plugging this back into \eqref{eq:IncnIinfty} yields the first inequality in \eqref{eq:I_error}, and the second inequality in \eqref{eq:I_error} follows from $\E[c_n \prod_{j \le n} U_j] = 2^{-n} c_n$.
\end{proof}

It remains to show that the value of $I_\infty$ equals $e^{-\gamma}$. This will be achieved using the recursive distributional equation \eqref{eq:recursion} with the help of Laplace transform $\phi(t) := \E[e^{-tS_\infty}]$, which is intrinsically related to our problem because
\begin{align}\label{eq:preLaplace}
I_\infty = \E[S_\infty^{-1}] = \E \left[ \int_0^\infty e^{-t S_\infty}\dd t\right] = \int_0^\infty \phi(t) \,\dd t
\end{align}
by Fubini's theorem. Let us first highlight that:
\begin{lemma}
The Laplace transform $\phi(t) := \E[e^{-tS_\infty}]$ satisfies
\begin{align}\label{eq:ODE}
te^t \phi(t)= \int_0^t \phi(v) \,\dd v, \qquad t \ge 0.
\end{align}
In particular,
\begin{align}\label{eq:preDickman}
I_\infty
= \lim_{t \to \infty} t e^t \phi(t).
\end{align}
\end{lemma}

\begin{proof}
From the recursive distribution equation \eqref{eq:recursion}, we have
\begin{align*}
\phi(t)
& := \E[e^{-tS_\infty}]
= \E[e^{-t(1+US_\infty)}]\\
& = e^{-t} \int_0^1 \E[e^{-tu S_\infty}]  \dd u 
= \frac{e^{-t}}{t} \int_0^t \E[e^{-v S_\infty}] \dd v.
\end{align*}
Hence $t e^t \phi(t) = \int_0^t \phi(v)\dd v $, as claimed. In particular $\lim_{t \to \infty} t e^t \phi(t) = \int_0^\infty \phi(v) \dd v = I_\infty$.
\end{proof}

\begin{proof}[Proof of \Cref{theo:problem1}]
Differentiating the equality \eqref{eq:ODE} yields $te^t\phi'(t)+(te^t)'\phi(t)=\phi(t)$, which may be rewritten as
\begin{align*}
\frac{\phi'(t)}{\phi(t)} = \frac{1- (te^t)'}{te^t} = \frac{e^{-t} - 1}{t}-1.
\end{align*}
Since $\phi(0)=1$, we then obtain
\begin{align*}
\log \phi(x) & = \log \phi(x) - \log \phi(0) = \int_0^x \frac{\dd}{\dd u}\Big[\log\phi(u)\Big]\dd{u}= \int_0^x \frac{\phi'}{\phi}(u)\dd{u}\\
& = \int_0^x \Big(\frac{e^{-u} - 1}{u}-1\Big)\dd{u} = \int_0^x \left(e^{-u}-1\right)\frac{\dd u}{u}-x\\
& = \left[ (e^{-u} - 1) \log u \right]_0^x + \int_0^x e^{-u} \log u \,\dd u - x.
\end{align*}
From Euler's identity for $\gamma$ \cite[Eq. (2.2.8)]{Lagarias},
\begin{align*}
\gamma = -\int_0^\infty e^{-u} \log u \,\dd u,
\end{align*}
we see as $x \to \infty$,
\begin{align*}
\log \phi(x) = - \log x - x - \gamma + o(1).
\end{align*}
Substituting this into \eqref{eq:preDickman}, we obtain $I_\infty
= \lim_{x \to \infty} x e^x \phi(x) = e^{-\gamma}$. 
Combining this with \Cref{lem:I_error}, we conclude that $I_n = e^{-\gamma} + O(2^{-n}(1+c_n))$.
\end{proof}
\section*{Acknowledgments}
Collaboration for this paper stems in part from a problem session during the `50 Years of Number Theory and Random Matrix Theory Conference' at the Institute for Advanced Study (IAS) in Princeton.
The authors would like to thank the conference organizers, as well as Paul Kinlaw and the anonymous referee for careful feedback. The first author is supported by funding from the European Research Council (ERC)
under the European Union's Horizon 2020 research and innovation programme (grant agreement No 851318). The second author is supported by a Clarendon Scholarship at the University of Oxford. The second and third author acknowledge support by travel grants from the Heilbronn Institute for Mathematical Research (HIMR).

\bibliographystyle{amsplain}

\providecommand{\bysame}{\leavevmode\hbox to3em{\hrulefill}\thinspace}
\providecommand{\MR}{\relax\ifhmode\unskip\space\fi MR }
\providecommand{\MRhref}[2]{%
  \href{http://www.ams.org/mathscinet-getitem?mr=#1}{#2}
}
\providecommand{\href}[2]{#2}

\end{document}